\documentclass[12pt]{dcds}
\usepackage{amsmath}
\usepackage{amssymb}

 \usepackage{paralist}
  \usepackage{graphics} 
 \usepackage{afterpage} 
 \usepackage{pgfplots}
\usepackage{tikz}
\usetikzlibrary{patterns, calc, arrows,decorations.pathreplacing,positioning}
\usetikzlibrary{matrix}
 
 \usepackage{float} 
 \usepackage{wrapfig}
\usepackage{lscape}
\usepackage[figuresright]{rotating}

\usepackage{graphicx}  \usepackage{epstopdf}
 \usepackage[colorlinks=true]{hyperref}
\hypersetup{urlcolor=blue, citecolor=red}

\definecolor{darkgreen}{rgb}{0,0.6,0}
\definecolor{darkblue}{rgb}{0.0,0.0,0.9}
\definecolor{darkred}{rgb}{0.8,0.0,0.0}
\definecolor{darkpurple}{rgb}{0.6,0.0,0.85}
\definecolor{darkorange}{rgb}{0.8,0.4,0.0}
\definecolor{orange}{rgb}{1.0,0.49,0.0}
\definecolor{palepink}{rgb}{1.0,0.83,0.83}
\definecolor{paleblue}{rgb}{0.9,0.9,1.0}
\definecolor{palepurple}{rgb}{1.0,0.9,1.0}
\definecolor{palegreen}{rgb}{0.9,1.0,0.9}
\definecolor{lesspalepink}{rgb}{1.0,0.75,0.75}
\definecolor{lesspaleblue}{rgb}{0.8,0.8,1.0}
\definecolor{lesspalepurple}{rgb}{1.0,0.85,1.0}
\definecolor{lesspalegreen}{rgb}{0.85,1.0,0.85}
\definecolor{evenlesspalepink}{rgb}{1.0,0.65,0.65}
\definecolor{evenlesspaleblue}{rgb}{0.7,0.7,1.0}
\definecolor{evenlesspalepurple}{rgb}{0.8,0.65,1.0}
\definecolor{evenlesspalegreen}{rgb}{0.75,1.0,0.75}
\definecolor{evenlesspaleorange}{rgb}{1.0,0.8,0.25}

\def\currentvolume{X}
 \def\currentissue{X}
  \def\currentyear{200X}
   \def\currentmonth{XX}
    \def\ppages{X--XX}
     \def\DOI{10.3934/xx.xx.xx.xx}

\newtheorem{theorem}{Theorem}[section]

\newtheorem{lemma}[theorem]{Lemma}

\newtheorem{claim}{Claim}[section]
\newenvironment{claimproof}[1]{\par\noindent\underline{Proof:}\space#1}{\hfill $\Diamond$}

\theoremstyle{definition}
\newtheorem{definition}[theorem]{Definition}
\newtheorem{remark}{Remark}

\def \diam {{\rm diam}\,}

\def \epsilon {\varepsilon}

\def \phi {\varphi}

\def \Gm {{\mathcal G}_m}

\def \J {{\mathcal J}}

\def \Jm {{\mathcal J}_m}

\def \Jn {{\mathcal J}_n}

\def \Jthinm {{\mathcal J}_{m,\scriptstyle{\mathrm{thin}}}}

\def \F {{\mathcal F}}

\def \Fm {{\mathcal F}_m}

\def \Fn {{\mathcal F}_n}

\def \H {{\mathcal H}}

\def \Hm {{\mathcal H}_m}

\def \Hm {{\mathcal H}_m}

\def \Hnt {\tilde{\mathcal H}_n}

\def \Hmt {\tilde{\mathcal H}_m}

\def \L {{\mathcal L}}

\def \Am {{\mathcal A}_{\infty, m}}

\def \C {\mathbb C}

\def \chat {{\widehat{\mathbb C}}}

\def \N {\mathbb N}

\def \nochat {{\mathbb N}_0 \times {\widehat {\mathbb C}}}

\def \Fsigma {{\mathrm F}_\sigma}

\def \Gdelta {{\mathrm G}_\delta}

\def \interior {\mathrm{int}}

\def \Pm {\{P_m \}_{m=1}^\infty}

\def \dt {{\mathrm d}t}

\def \dz {{\mathrm d}z}

\title[TT Julia Sets]{Thick-Thin non-Autonomous Julia Sets}

\renewcommand{\shorttitle}{Thick-Thin non-Autonomous Julia Sets}

\author{Mark Comerford, Hiroki Sumi}



\thanks{
The authors wish to thank Rich Stankewitz of Ball State University for his kind help with earlier versions of this manuscript.}
\thanks{The second author is
partially supported by JSPS Grant-in-Aid for Scientific Research (B) 
Grant Number JP24K00526.
}

\email{mcomerford@math.uri.edu}
\email{sumi@math.h.kyoto-u.ac.jp}

\begin{document}

\maketitle
\centerline{\scshape Mark Comerford}
\medskip
{\footnotesize
 \centerline{Department of Mathematics}
   \centerline{University of Rhode Island}
   \centerline{5 Lippitt Road, Room 102F}
   \centerline{Kingston, RI 02881, USA}
} 

\medskip


\medskip

\centerline{\scshape Hiroki Sumi }
\medskip
{\footnotesize
 \centerline{Division of Mathematical and Information Sciences}
 \centerline{Graduate School of Human and Environmental Studies}
   \centerline{Kyoto University}
   \centerline{Yoshida-nihonmatsu-cho, Sakyo-ku}
   \centerline{Kyoto 606-8501, Japan}
   \centerline{https://www.math.h.kyoto-u.ac.jp/users/sumi/index.html}
} 

\bigskip

\begin{abstract}Hereditarily non uniformly perfect (HNUP) sets were introduced by Stankewitz, Sugawa, and Sumi in \cite{SSS} who gave several examples of such sets based on Cantor set-like constructions using nested intervals. For non-autonomous iteration where one considers compositions of polynomials from a sequence which is in general allowed to vary, the Julia set is uniformly perfect for all sequences with suitably bounded coefficients, while 
Comerford, Stankewitz and Sumi showed in \cite{CSS} that for certain sequences of polynomials with unbounded coefficients, it is possible to have Julia sets which are HNUP. 
In this manuscript we give an example of a non-autonomous polynomial sequences whose Julia sets lie in between these two extremes in that they are not uniformly perfect, but also not HNUP. In addition we show that these Julia sets can be expressed as a `thick-thin' decomposition consisting of a ${\mathrm F}_\sigma$ subset which is a countable union of uniformly perfect sets and 
a ${\mathrm G}_\delta$ subset which is HNUP.
\end{abstract}

\section{Introduction}

Our paper is concerned with non-autonomous iteration of complex polynomials. This subject was started by Fornaess and Sibony \cite{FS} in 1991 and by Sester, Sumi, and others who were working in the closely related area of skew-products \cite{Ses, Sumi1, Sumi2, Sumi3, Sumi4}. There is also an extensive literature in the real variables case which is mainly focused on topological dynamics, chaos, and difference equations, e.g., \cite{Balibrea, CL}.  

We begin with the basic definitions we need in order to state the main theorems of this paper. In the following sections we then prove these theorems together with some supporting results and then make a few concluding remarks.

\subsection{Polynomial Sequences}

Let $\Pm$ be a sequence of polynomials where each $P_m$ has degree $d_m \ge 2$. For each $0 \le m$, let $Q_m$ be the composition $P_m \circ \cdots \cdots \circ P_2 \circ P_1$ (where for convenience we set $Q_0 = Id$)
and, for each $0 \le m \le n$, let $Q_{m,n}$ be the composition $P_n \circ \cdots \cdots \circ P_{m+2} \circ P_{m+1}$ (where we let each $Q_{m,m}$ be the identity). Such a sequence can be thought of in terms of the sequence of iterates of a skew product on $\chat$ over the non-negative integers $\N_0$ or, equivalently, in terms of the sequence of iterates of a mapping $F$ of the set $\nochat$ to itself, given by $F(m, z) := (m+1, P_{m+1}(z))$.
Let the degrees of these compositions $Q_m$ and $Q_{m, n}$ be $D_m$ and $D_{m,n}$ respectively so that $D_m = \prod_{i=1}^m d_i$, $D_{m,n} = \prod_{i=m+1}^n d_i$.

For each $m \ge 0$ define the \emph{$m$th iterated Fatou set} $\Fm$ by
\[ \Fm = \{z \in \chat : \{Q_{m,n}\}_{n=m}^\infty \;
\mbox{is normal on some neighbourhood of}\; z \}\]
where we take our neighbourhoods with respect to the spherical topology on $\chat$. We then define the \emph{$m$th iterated Julia set} $\Jm$ to be the complement $\chat \setminus \Fm$. At time $m =0$ we call the corresponding iterated Fatou and Julia sets simply the \emph{Fatou} and \emph{Julia sets} for our sequence and designate them by $\F$ and $\J$ respectively.

One can easily show that the iterated Fatou and Julia sets are completely invariant in the following sense.

\begin{theorem}\label{ThmCompInvar}
For each $0 \le m \le n$, $Q_{m,n}(\Fm) = \Fn$ and $Q_{m,n}(\Jm) = \Jn$ with components of $\Fm$ being mapped surjectively onto components of $\Fn$.
\end{theorem}

An important special case is when we have an integer $d \ge 2$ and real numbers $M \ge 0$, $K \ge 1$ for which our sequence $\Pm$ is such that
\[P_m(z) = a_{d_m,m}z^{d_m} + a_{d_m-1,m}z^{d_m-1} + \cdots \cdots +
a_{1,m}z + a_{0,m}\]
is a polynomial of degree $2 \le d_m \le d$ whose coefficients satisfy
\[\frac{1}{K} \le |a_{d_m,m}| \le K,\quad m \ge 1, \:\:\quad \quad |a_{k,m}| \le M,\quad m \ge 1,\:\:\: 0 \le k \le d_m -1. \vspace{.1cm}\]
Such sequences are called \emph{bounded sequences of polynomials} or simply \emph{bounded sequences} (see e.g. \cite{Com1, Com2}), this definition being a slight generalization of that originally made by Fornaess and Sibony in \cite{FS} who considered bounded sequences of monic polynomials.

A very useful result was proved by B\"uger in \cite{Bug} concerning self-similarity. B\"uger considered sequences $\Pm$ for which there existed a hyperbolic domain $U \subset \chat$ which was invariant in the sense that $\infty \in U$, $P_m(U) \subset U$ for each $m \ge 0$, and $Q_m \to \infty$ locally uniformly on $U$.

\begin{theorem}[\cite{Bug} Theorem 2]
\label{SelfSimilarity}
Let $\Pm$ be a sequence for which there exists an invariant domain $U$ as above. Then, for any $m \ge 0$,
any $z \in \Jm$, and any neighbourhood $N$ of $z$, the following three equivalent statements hold:

\vspace{-.1cm}
\begin{enumerate}
\item The set $\cup_{n \ge m}Q_{m,n}^{-1}(Q_{m,n}(z))$ is dense in $\Jm$,
\vspace{.2cm}
\item There exists $n_0 \in \N$ for which $Q_{m,n_0}^{-1}(Q_{m,n_0}(N)) \supset \Jm$, 
\vspace{.2cm}
\item There exists $n_0 \in \N$ for which $Q_{m,n_0}(N) \supset {\mathcal J}_{n_0}$.
\end{enumerate}
\end{theorem}

\vspace{-.1cm}
We remark that the original version of this as stated by B\"uger only had part \emph{2.} of the statement. However, it is easy to see that \emph{1.} and \emph{3.} also hold and both are very useful in practice. One of the great strengths of this result is that the sequence $\Pm$ does not need to be bounded, which will be of crucial importance to us.

In what follows, for $z \in \C$ and $r > 0$, we use the notation ${\mathrm D}(z, r)$ for the open disc with centre $z$ and radius $r$, while the corresponding closed disc and boundary circle will be denoted by
$\overline{\mathrm D}(z,r)$ and ${\mathrm C}(z, r)$ respectively. For $z \in \C$ and $0 < r< R$, we use ${\mathrm A}(z, r,R)$ for the round annulus $\{w:r < |w-z| < R\}$ with centre $z$, inner radius $r$, and outer radius $R$, while we use $\overline {\mathrm A}(z, r,R)$ for the corresponding closed annulus. 

For $z \in \C$ and $\theta_1, \theta_2 \in [0,2\pi)$, we denote by ${\mathrm S}(z, \theta_1. \theta_2)$ the closed sector $\{w: \theta_1 \le  \arg(w-z) \le \theta_2\}$ where, in the case $\theta_2 < \theta_1$ we include angles in the ranges $[\theta_1, 2\pi)$ and $[0, \theta_2]$
(in the degenerate case where $\theta_1 = \theta_2$, the sector is simply a half-line at $z$ which makes an angle $\theta_1$ with the positive real axis). The \emph{aperture} of such a sector is then given by $\theta_2 - \theta_1$ in the case $\theta_1 \le \theta_2$ and $2\pi - \theta _1 + \theta_2$ in the case $\theta_2 < \theta_1$. Given this, for a point $z \in \C$ and a closed set $X \subset \C$ with $z \notin X$, we define the \emph{angle subtended by $X$ at $z$} to be the aperture of the sector with the smallest possible aperture which contains $X$. 

\subsection{Uniformly Perfect and Hereditarily non Uniformly Perfect Sets}\label{SectHNUP}

We call a doubly connected domain $A$ in $\C$ that can be conformally mapped onto a true (round) annulus $\mathrm{A}(z,r,R)$, for some $0<r<R$, a \emph{conformal annulus} with the \emph{modulus} of $A$ given by $\textrm{mod }A=\log(R/r)$, noting that $R/r$ is uniquely determined by $A$ (see, e.g., the version of the Riemann mapping theorem for multiply connected domains in \cite{Ahl}).

\begin{definition} \label{sepann}
A conformal annulus $A$ is said to
\emph{separate} a set $F \subset \C$ if $F \cap A = \emptyset$ and $F$ intersects both components of $\C \setminus
A$.
\end{definition}

\begin{definition} \label{updef}
A compact subset $F \subset \C$ with two or more points is \emph{uniformly perfect} if
there exists a uniform upper bound on the moduli of all conformal annuli which separate $F$.
\end{definition}

We observe that it is trivial to show that any connected subset of $\C$ consisting of more than two points is uniformly perfect. 
An extreme version of failing to be uniformly perfect is the concept of hereditarily non uniformly perfect which was introduced in
\cite{SSS} and can be thought of as a thinness criterion for sets:
\begin{definition}
A compact set $E$ is called \textit{hereditarily non uniformly perfect} (HNUP) if no subset of $E$ is uniformly perfect.
\end{definition}

Often a compact set is shown to be HNUP by showing it satisfies the following stronger property of \emph{pointwise thinness}.  This is done in several examples in~\cite{SSS,CSS,FalkStankewitz}, and will also be done in this paper.
A set $E \subset \C$ is \textit{pointwise thin at} $z \in E$ if there exists a sequence of conformal annuli $A_n$ each of which separates $E$, has $z$ in the bounded component of its complement, and such that $\textrm{mod }A_n \to +\infty$ while the Euclidean diameter of $A_n$ tends to zero.  A set $E \subset \C$ is called \emph{pointwise thin} when it is pointwise thin at each of its points. This then leads to the following definition. 

\begin{definition}
\label{JPointwiseThin}
Let $\Pm$ be a sequence of polynomials. For each $m \ge 0$ we define the \emph{pointwise thin} part $\Jthinm$ of the iterated Julia set $\Jm$ as the set of points in $\Jm$ where $\Jm$ is pointwise thin. 
\end{definition}

Specifically, $\Jthinm$ is the set of points $z \in \Jm$ for which  there exists a sequence of conformal annuli $A_n$ each of which separates $\Jm$, has $z$ in the bounded component of its complement, and such that $\textrm{mod }A_n \to +\infty$ while the Euclidean diameter of $A_n$ tends to zero. In \cite{CFSS, CSS} we constructed examples where the whole Julia set was pointwise thin. The focus of this paper is examples where for each $m \ge 0$ the set $\Jthinm$ is non-empty, but is not the whole of the iterated Julia set $\Jm$.

\subsection{Statements of the Main Result}

It is well known that the Julia set for a single rational function (of degree at least $2$) is uniformly perfect (see \cite{Hi, MdR}) as are the iterated Julia sets for a 
bounded sequence of polynomials (see Theorem 1.6 of~\cite{Sumi3}]). Furthermore, by Theorem 3.1 of~\cite{RS3}, the Julia set of any non-elementary rational semigroup $G$, which is allowed to contain or even consist of M\"obius maps, is uniformly perfect when there is a uniform upper bound on the Lipschitz constants (with respect to the spherical metric) of the generators of $G$.

The construction of the sequences of polynomials we consider in this paper will therefore require unbounded sequences of polynomials, which will be quadratic polynomials. For a sequence of polynomials $\Pm$, we say the collection of sets $\{X_m, m \ge 0\}$ is \emph{forward invariant} if $P_{m+1}(X_m) = X_{m+1}$ for each $m \ge 0$ and \emph{backward invariant} if $P_{m+1}^{-1}(X_{m+1}) = X_m$ for each $m \ge 0$. If $\{X_m, m \ge 0\}$ is both forward and backward invariant, we say that it is \emph{completely invariant}. Note that Theorem \ref{ThmCompInvar} states that the iterated Julia and Fatou sets form completely invariant collections.

\begin{theorem}\label{MainTh1}
There exists an unbounded sequence of quadratic polynomials $\Pm$ where each $P_m$ is of the form $z^2 +c_m$ for which the iterated Julia sets $\Jm$ are not uniformly perfect but not HNUP. More precisely, for each $m \ge 0$ $\Jm = \Gm \cup \Hm$ where 

\begin{enumerate}
\item $\Gm$ and $\Hm$ are disjoint, non-empty, and each form completely invariant collections,

\vspace{.2cm}
\item $\Gm$ is a dense relatively $\Gdelta$ subset of $\Jm$ and $\Jthinm = \Gm$, so that in particular $\Gm$ is pointwise thin,

\vspace{.2cm}
\item $\Hm$ is a dense $\Fsigma$ subset of $\Jm$ and is a countable increasing union of closed sets each of which is uniformly perfect,

\vspace{.2cm}
\item $\overline {\mathcal H}_m = \Jm$ is not uniformly perfect,

\vspace{.2cm}
\item $\Hm$ consists of countably infinitely many connected components (in $\Jm$), each component of which bounds at least one bounded component of the Fatou set $\Fm$ and is uniformly perfect. 
\end{enumerate}

\end{theorem}

Comparing this with the earlier work of Comerford, Stankewitz, and Sumi \cite{CSS} where we constructed examples where the whole of the iterated Julia sets were HNUP (see also~\cite{CFSS}), we remark that, due to the iterated Julia sets in Theorems~\ref{MainTh1} consisting of two such different pieces, they are, in a certain sense, less self-similar than the earlier HNUP Julia sets and the construction is correspondingly a good deal more complicated. In particular, there are critical points which pass close to the iterated Julia sets, which requires careful handling, particularly when it comes to controlling distortion, which is crucial here, as it is in so many examples in the area of non-autonomous polynomial dynamics.
 
\section{Survival Sets and Julia Sets}\label{Survival}

The polynomial sequence in Theorem \ref{MainTh1} is necessarily unbounded and so we cannot use the 
usual escape radius criterion to characterize the iterated filled Julia sets. However, these sequences are constructed in a similar way to satisfy certain conditions designed to make it relatively easy to locate the iterated Julia and filled Julia sets.

To be precise, we start with two fixed sequences $\{a_k\}_{k=1}^\infty$, $\{b_k\}_{k=1}^\infty$ in $\C^\N$ with $|b_k| > 4$ for each $k \ge 1$. Given this, we define a sequence $\{m_k\}_{k=1}^\infty$ of natural numbers for which we have, for each $k \ge 1$,
\begin{eqnarray}
\label{Invariance1} 
2^{2^{m_k}} - |a_k| &>& 2^k,\\
\label{Invariance2}
(2^{2^{m_k}} - |a_k|)^2 - |b_k| & > & 2^k.
\end{eqnarray}

Since $a_k$ and $b_k$ are fixed, we note that we can certainly choose $m_k$ sufficiently large so that this occurs. Now set $M_0 = 0$, $M_k = \sum_{j=1}^k {(m_j + 1)}$ for each $k \ge 1$, and define a sequence of quadratic polynomials $\Pm$ by
\vspace{0.2cm}
\begin{eqnarray}
\label{PmkDef}
P_m = \left \{ \begin{array}{r@{\: ,\quad}l}
z^2 + a_k & \mbox{if} \:\; m=M_k -1\quad \mbox{for some} \:\;k \,\ge 1
\vspace{.2cm}\\
z^2 + b_k & \mbox{if} \:\; m=M_k \quad \mbox{for some} \:\;k \,\ge 1
\vspace{.2cm}\\

z^2 &  \mbox{otherwise.} \\
\end{array} \right .  
\end{eqnarray}

\begin{remark} \label{Remark1}
For each $k \ge 1$, since $|b_k| > 4$, the inverse image of $\overline{\mathrm D}(0,2)$ under $P_{M_k} = z^2 + b_k$ consists of two components whose maximum distance from the origin is $\sqrt{|b_k| + 2}$ and which then lie within a disc about $a_k$ whose radius is bounded above by $|a_k| + \sqrt{|b_k| + 2}$. On the other hand, the image of $\overline{\mathrm D}(0,2)$ under $Q_{M_{k-1},M_k -1} = z^{2^{m_k}} + a_k$ is the disc ${\mathrm D}(a_k, 2^{2^{m_k}})$ and it follows immediately from \eqref{Invariance2} above that $(2^{2^{m_k}} - |a_k|)^2 - |b_k| >  2$. Since by \eqref{Invariance1} $2^{2^{m_k}} - |a_k| > 2^k > 0$, we then immediately obtain that  $2^{2^{m_k}} >  |a_k| + \sqrt{|b_k| + 2}$.

Thus, both preimage components of $\overline{\mathrm D}(0,2)$ under $P_{M_k}$ lie inside the disc ${\mathrm D}(a_k, 2^{2^{m_k}})$ and, on taking the inverse image under $Q_{M_{k-1},M_k -1} = z^{2^{m_k}} + a_k$ on this disc, it follows that all of the preimage components of $\overline{\mathrm D}(0,2)$ under $Q_{M_{k-1},M_k}$ will lie inside $\overline{\mathrm D}(0,2)$, i.e., 

 \begin{align}
 \label{Invariance3}
 Q_{M_{k-1},M_k}^{-1}( \overline{\mathrm D}(0,2)) \subset \overline{\mathrm D}(0,2).
 \end{align}
 
 \end{remark}

\begin{definition}
\label{SkDef}
Given a sequence $\Pm$ as above, for each $k \ge 1$, we define the \emph{$k$th survival set ${\mathcal S}_k$ at time $0$} by

\vspace{-.6cm}
\begin{equation}\label{Sk}
{\mathcal S}_k = Q_{M_k}^{-1}(\overline {\mathrm D}(0, 2)) = Q_{M_0, M_{1}}^{-1} ( \cdots \cdots (Q_{M_{k-1},M_{k}}^{-1}(\overline {\mathrm D}(0, 2)))\cdots).
\end{equation}   
\end{definition}

Note that it follows immediately from \eqref{Invariance3} that these sets are decreasing in $k$ and are all contained in $\overline{\mathrm D}(0,2)$. Before we show how these sets are related to the Julia and Fatou sets, we first make some additional remarks.

\begin{remark} \label{Remark2}
\begin{enumerate}[(a)]
\item\label{Remark2a} Using the fundamental theorem of calculus, for any complex number $b$, the maximum diameter of each component of the preimage of $\overline {\mathrm D}(0, 2)$ under $z^2 + b$ is $2$.

\item \label{Remark2b} We also note that for any $m \ge 1$, any complex number $a$, and any set $X$ of diameter $\le 4$, the diameter of each component of the preimage of $X$ under the polynomial $z^{2^m}+a$ is at most $4$. To see this, if we let $z$ be any point of $X$, then clearly $X \subset \overline {\mathrm D}(z, 4)$ and the conclusion then follows using calculus in a similar way to above since $8^{1/2^m} <  4$.

\item \label{Remark2c}  Combining parts ~(\ref{Remark2a}) and~(\ref{Remark2b}) with \eqref{Invariance3} above, we see that, for any $0 \le m \le M_k$, each component of $Q_{m , M_k}^{-1}(\overline {\mathrm D}(0, 2))$ has Euclidean diameter at most $4$. From this, we see that the spherical diameters of these preimage components are uniformly bounded above away from $\pi = \mathrm{diam}^{\#} \chat$, the spherical diameter of the Riemann sphere $\chat$.

\end{enumerate}
\end{remark}

We then have the following result. 

\begin{theorem}\label{ThmJm}
For a sequence $\Pm$ as above, we have $\mathcal{J}=\partial \bigcap_{k \ge 1} {\mathcal S}_k$ while 
$\interior \bigcap_{k \ge 1} {\mathcal S}_k \subset \F$. Additionally,  for each $m \ge 0$,

\begin{equation}
\label{Boundary}
{\mathcal J_m} = Q_m \left( \partial \bigcap_{k \ge 1} {\mathcal S}_k \right ) =  \partial \bigcap_{k \ge 1} Q_m({\mathcal S}_k),
\end{equation}

while 

\vspace{-.2cm}
\begin{equation}
\label{Interior}
Q_m \left( \interior \bigcap_{k \ge 1} {\mathcal S}_k \right ) =  \interior \bigcap_{k \ge 1} Q_m({\mathcal S}_k) \subset \Fm.
\end{equation}
\end{theorem}

\begin{proof}
We first prove the result for $m = 0$ (recalling that $Q_0 = Id$), basing our proof on showing that $\bigcap_{k \ge 1} {\mathcal S}_k$ is
precisely the set of points whose orbits do not escape locally uniformly to infinity.

Suppose first that $z \notin \bigcap_{k \ge 1} {\mathcal S}_k$, i.e., $|Q_{M_k}(z)| > 2$ for some $k$. From \eqref{Invariance1}, \eqref{Invariance2} it then follows 
 that $Q_{M_j-1}(z)$ and $Q_{M_j}(z)$ both tend to infinity as $j \to \infty$.  Note that, for each $j \ge k$ and $0 \leq N \le m_{j+1} -1$, since $Q_{M_j,M_j+N}(z)=z^{2^N}$, we see that $|Q_{M_j+N}(z)|=|Q_{M_j,M_j+N}(Q_{M_j}(z))|\ge |Q_{M_j}(z)|$.  From this it clearly follows that $Q_{m}(z)\to \infty$ as $m \to \infty$, and at a rate which is locally uniform, whence we must have that $z \in {\mathcal F}$.

On the other hand, let $z \in  \partial \bigcap_{k \ge 1} {\mathcal S}_k  \subseteq \bigcap_{k \ge 1} {\mathcal S}_k$.  Then $|Q_{M_k}(z)| \le 2$ for every $k$, while from above, $z$ is approached by points whose orbits escape locally uniformly to infinity. From this we see that $z \in \J$ as desired. 

Finally, suppose $z \in \interior \bigcap_{k \ge 1} {\mathcal S}_k$ and let $U$ be the component of $\interior \bigcap_{k \ge 1} {\mathcal S}_k$ containing $z$. By Remark \ref{Remark2} (\ref{Remark2c}) above, we see that the spherical diameters of the sets $Q_m(U)$ are uniformly bounded above away from {\color{darkgreen}$\pi$} from which it follows that $U \subset \F$ in view of Theorem 3.3.5 in \cite{Bear} (which is a version of Montel's theorem). The result for $m=0$ then follows on combining the three observations above.

To prove the result for arbitrary $m$ and establish \eqref{Boundary}, \eqref{Interior}, we fix $m > 0$ and apply $Q_m$ to the $m=0$ case proved above. By complete invariance (Theorem~\ref{ThmCompInvar}) $Q_m(\J) = \Jm$, $Q_m(\F) = \Fm$
whence from above ${\mathcal J_m} = Q_m \left( \partial \bigcap_{k \ge 1} {\mathcal S}_k \right )$ and $Q_m \left( \interior \bigcap_{k \ge 1} {\mathcal S}_k \right ) \subset \Fm$. For the remainder of \eqref{Boundary}, \eqref{Interior}, if $k \ge 1$ is such that $M_k \ge m$, then by Definition \ref{SkDef}

\vspace{-.6cm}
$$Q_m( {\mathcal S}_k) = Q_m(Q_m^{-1}\circ Q_{m, M_k}^{-1} (\overline {\mathrm D}(0, 2))) = 
Q_{m, M_k}^{-1} (\overline {\mathrm D}(0, 2))\vspace{.1cm}$$
so that 
\vspace{.1cm}
$$Q_m^{-1}(Q_m( {\mathcal S}_k)) = Q_m^{-1}(Q_{m, M_k}^{-1} (\overline {\mathrm D}(0, 2))) = 
Q_{M_k}^{-1}(\overline {\mathrm D}(0, 2)) = {\mathcal S}_k.$$

As already observed, by \eqref{Invariance3}, the sets ${\mathcal S}_k$ and thus $Q_m({\mathcal S}_k)$ are decreasing in $k$ and so for any $k_0 \ge 1$ we can take $\cap_{k \ge 1} {\mathcal S}_k = \cap_{k \ge k_0} {\mathcal S}_k$ and $\cap_{k \ge 1} Q_m({\mathcal S}_k) = \cap_{k \ge k_0} Q_m({\mathcal S}_k)$. If we then let $k_0$ be the smallest natural number for which $M_{k_0} \ge m$, from above we have

\vspace{-.6cm}
\begin{eqnarray*}
Q_m^{-1}(\cap_{k \ge 1}Q_m({\mathcal S}_k))&=&
Q_m^{-1}(\cap_{k \ge k_0}Q_m({\mathcal S}_k))\nonumber\\ 
&=& \cap_{k \ge k_0} Q_m^{-1}(Q_m( {\mathcal S}_k))\nonumber\\\
&=&  \cap_{k \ge k_0} {\mathcal S}_k\nonumber\\\
&=&  \cap_{k \ge 1} {\mathcal S}_k
\end{eqnarray*}
from which it follows immediately that 
\vspace{.2cm}
\begin{equation*}
Q_m^{-1}(\C \setminus \cap_{k \ge 1}Q_m({\mathcal S}_k)) = \C \setminus \cap_{k \ge 1} {\mathcal S}_k.\vspace{.2cm}
\end{equation*}
If we apply $Q_m$ to each of the above, we immediately obtain

\vspace{-.3cm}
\begin{equation}
\label{QmSk}
Q_m( \cap_{k \ge 1} {\mathcal S}_k) = \cap_{k \ge 1}Q_m({\mathcal S}_k),
\end{equation}

\vspace{-.3cm}
\begin{equation}
\label{QmSkc}
Q_m(\C \setminus \cap_{k \ge 1} ({\mathcal S}_k) = \C \setminus \cap_{k \ge 1}Q_m({\mathcal S}_k).
\end{equation}

Since $\partial (\cap_{k \ge 1} {\mathcal S}_k)$ is precisely the set of points which are adherent to both $\cap_{k \ge 1} {\mathcal S}_k$ and $\C \setminus \cap_{k \ge 1} {\mathcal S}_k$, it follows by \eqref{QmSk} and \eqref{QmSkc} above that $Q_m (\partial  (\cap_{k \ge 1} {\mathcal S}_k)) \subseteq \partial (\cap_{k \ge 1}Q_m({\mathcal S}_k))$. On the other hand, by \eqref{QmSk}, each point of $\partial (\cap_{k \ge 1}Q_m({\mathcal S}_k))$ is the image under $Q_m$ of a point in $\cap_{k \ge 1} {\mathcal S}_k$.
However, by the open mapping theorem, 
\vspace{.2cm}
\begin{equation}
\label{openmapping}
Q_m \left( \interior \bigcap_{k \ge 1} {\mathcal S}_k \right ) \subseteq  \interior \,Q_m \left (\bigcap_{k \ge 1} {\mathcal S}_k \right ) \subseteq \interior \bigcap_{k \ge 1} Q_m({\mathcal S}_k)
\end{equation}
and it follows that $\partial (\cap_{k \ge 1}Q_m({\mathcal S}_k)) \subseteq Q_m (\partial  (\cap_{k \ge 1} {\mathcal S}_k))$ which finishes the proof of \eqref{Boundary}. 

For \eqref{Interior}, by \eqref{openmapping} above $Q_m \left( \interior \bigcap_{k \ge 1} {\mathcal S}_k \right ) \subseteq  \interior \bigcap_{k \ge 1} Q_m({\mathcal S}_k)$. On the other hand, again by \eqref{QmSk}, each point of $\interior (\cap_{k \ge 1}Q_m({\mathcal S}_k))$ is the image under $Q_m$ of a point in $\cap_{k \ge 1} {\mathcal S}_k$ and it follows by combining this with \eqref{Boundary} that $\interior (\cap_{k \ge 1}Q_m({\mathcal S}_k)) \subseteq Q_m \left( \interior \bigcap_{k \ge 1} {\mathcal S}_k \right )$ which finishes the proof of \eqref{Interior} and of the result. \end{proof}

We now prove several lemmas which will be of use to us later in our manuscript. To begin the construction of the examples in the theorem, we first let $\{b_k\}_{k=1}^\infty$ be any fixed but arbitrary sequence of negative real numbers with $b_k < -6$ for every $k$ and such that $b_k \to -\infty$ as $k \to \infty$ (note that we need a sequence which converges to $-\infty$ rather than just an unbounded sequence - the reason why is given towards the end of the proof of Theorem \ref{MainTh1}). The sequences $\{a_k\}_{k=1}^\infty$ and $\{m_k\}_{k=1}^\infty$ for now we consider to be fixed but arbitrary, except that the sequence $\{m_k\}_{k=1}^\infty$ is chosen so that the numbers $m_k$ are sufficiently large (depending on $|a_k|$ and $|b_k|$) so that 
 the invariance conditions \eqref{Invariance1}, \eqref{Invariance2} and thus also \eqref{Invariance3} above are met. 

We now define two families of sets  $G^k_m$, $H^k_m$ on which we will be relying heavily in the rest of what follows.

\begin{definition}
\label{GmkHmkDef}Let $k \ge 1$. We define the sets $G^k_m$, $H^k_m$ for each $m \ge 0$ as follows: 

If $m = M_k -1$, let  $G^k_{M_k - 1}$, $H^k_{M_k - 1}$ be the two components of $P_{M_k}^{-1}(\overline {\mathrm D}(0, 2))$ at time $M_k -1$ which lie to the right and left, respectively, of the origin about the points $\pm \sqrt{-b_k}$.  

For $0 \le m \le M_k - 2$, we define $G^k_m$, $H^k_m$ inductively working backwards from time $M_k - 2$ by 
\vspace{0.2cm}
\[G^k_m = \left \{ \begin{array}{r@{\: ,\quad}l}
P_{m+1}^{-1}(G^k_{m+1})& \mbox{if} \:\; m \ne M_j - 1 \:\: \mbox{for any} \:\,1 \,\le j < k
\vspace{.2cm}\\
P_{m+1}^{-1}(G^k_{m+1}) \cap G^j_{M_j -1}& \mbox{if} \:\; m=M_j - 1 \:\: \mbox{for some}  \:\,1 \,\le j < k, \end{array} \right .  \]

\vspace{-0.2cm}
\[H^k_m = \left \{ \begin{array}{r@{\: ,\quad}l}
P_{m+1}^{-1}(H^k_{m+1})& \mbox{if} \:\; m \ne M_j -1 \:\: \mbox{for any} \:\,1 \,\le j < k
\vspace{.2cm}\\
P_{m+1}^{-1}(H^k_{m+1}) \cap H^j_{M_j -1}& \mbox{if} \:\; m=M_j -1 \:\: \mbox{for some}  \:\,1 \,\le j < k. \end{array} \right .  \]

Finally, if $m \ge M_k -1$, we let $G^k_m$, $H^k_m$ be forward images of these sets under $Q_{M_k -1,m}$, i.e. 

\[ G^k_m = Q_{M_k -1,m}(G^k_{M_k - 1}), \qquad H^k_m = Q_{M_k -1,m}(H^k_{M_k - 1}).\]
\end{definition}

Before embarking on the proof of the next result which gives some of the basic properties of these sets, we direct the reader's attention to Figure \ref{MainTh1Picture} which shows a picture of these sets in the specific context of the proof of Theorem \ref{MainTh1}.

\begin{lemma}
\label{GmkHmkDecr}
For each $m \ge 0$, $k \ge 1$ we have: 

\begin{enumerate} 
\item The sets $G^k_m$, $H^k_m$ are compact, non-empty, and decreasing in $k$,

\vspace{.2cm}
\item $P_{m+1}(G^k_m) = G^k_{m+1}$ and $P_{m+1}(H^k_m) = H^k_{m+1}$.  

\end{enumerate}
\end{lemma}

\begin{proof}
For the first part \emph{1.}, the compactness (and non-emptiness) of the sets $G^k_m$, $H^k_m$ is immediate, given that polynomials are proper, continuous, and surjective mappings on $\C$. 

For the decreasing property in \emph{1.} above, we will show this for the sets $H^k_m$; the argument for the sets $G^k_m$ is then identical. So let $1 \le k_1 < k_2$ so that $M_{k_1} < M_{k_2}$. We now consider the following three cases:

{\bf Case (i):} $M_{k_1} < M_{k_2} \le m$. 
Using \eqref{Invariance3} repeatedly,
\begin{eqnarray*}
Q_{M_{k_1}-1, M_{k_2}}(H^{k_1}_{M_{k_1} - 1}) &=& Q_{M_{k_1}, M_{k_2}}(\overline {\mathrm D}(0, 2))\\
&\supset& \overline {\mathrm D}(0, 2)\\
&=& Q_{M_{k_2}-1, M_{k_2}}(H^{k_2}_{M_{k_2} - 1}) = P_{M_{k_2}}(H^{k_2}_{M_{k_2} - 1})
\end{eqnarray*}
and the conclusion in this case follows on applying the composition $Q_{M_{k_2},m}$ to the above.

{\bf Case (ii):} $M_{k_1} \le m <  M_{k_2}$. 
Again using \eqref{Invariance3} repeatedly,
\vspace{.1cm}
\begin{eqnarray}
\label{IntermdiateCase}
Q_{M_{k_1}, M_{k_2}-1}^{-1}(H^{k_2}_{M_{k_2} - 1}) &\subset& Q_{M_{k_1}, M_{k_2}}^{-1}(\overline {\mathrm D}(0, 2)) \nonumber\\
&\subset& \overline {\mathrm D}(0, 2) \nonumber\\
&=& Q_{M_{k_1}-1, M_{k_1}}(H^{k_1}_{M_{k_1} - 1})\nonumber\\ 
&=& P_{M_{k_1}}(H^{k_1}_{M_{k_1} - 1}).
\end{eqnarray}
In addition (making use of the definition of the set $H^{k_2}_m$ and remembering that we take just one preimage branch from the times $M_j$ to the times $M_j-1$, $k_1 < j < k_2$) we have

\vspace{-.5cm}
\begin{eqnarray*}
Q_{M_{k_1}, m}(Q_{M_{k_1}, M_{k_2}-1}^{-1}(H^{k_2}_{M_{k_2} - 1})) &=&  
Q_{M_{k_1}, m}(Q_{M_{k_1}, m}^{-1}(Q_{m, M_{k_2}-1}^{-1}(H^{k_2}_{M_{k_2} - 1})))  \\ 
&=&
Q_{m, M_{k_2}-1}^{-1}(H^{k_2}_{M_{k_2} - 1}) \supset H^{k_2}_m,
\end{eqnarray*}
while  
$$Q_{M_{k_1}, m}(P_{M_{k_1}}(H^{k_1}_{M_{k_1} - 1})) =  Q_{M_{k_1}-1, m}(H^{k_1}_{M_{k_1} - 1}) = H^{k_1}_m$$ 

and the conclusion in this case then follows by combining the two observations above together with \eqref{IntermdiateCase}. 

{\bf Case (iii):} $m < M_{k_1} < M_{k_2}$. 
Yet again using \eqref{Invariance3} repeatedly and also the definition of $H^{k_2}_{M_{k_1}}$ (again remembering that we take just one preimage branch at the times $M_j-1$, $k_1 < j < k_2$)
\begin{eqnarray*}
P_{M_{k_1}}(H_{M_{k_1}-1}^{k_2}) = H^{k_2}_{M_{k_1}} &\subset& Q_{M_{k_1}, M_{k_2}-1}^{-1}(H^{k_2}_{M_{k_2} - 1})\\ 
&\subset& Q_{M_{k_1}, M_{k_2}-1}^{-1}(P_{M_{k_2}}^{-1}(\overline {\mathrm D}(0, 2)))\\
&=& Q_{M_{k_1}, M_{k_2}}^{-1}(\overline {\mathrm D}(0, 2))\\
&\subset& \overline {\mathrm D}(0, 2)\\
&=& Q_{M_{k_1}-1, M_{k_1}}(H^{k_1}_{M_{k_1} - 1}) = P_{M_{k_1}}(H^{k_1}_{M_{k_1} - 1}).
\end{eqnarray*}
 The conclusion in this case follows from taking preimages in turn under the polynomials $P_{M_{k_1}}, P_{M_{k_1}-1}, \ldots \ldots, P_m$ while using the definition of the sets $H^k_m$ to take just one preimage branch from the times $M_j$ to the times $M_j-1$, where $m < M_j \le M_{k_1}$ while the two inverse branches of $P_{M_j}$ are univalent since $H_{M_j}^j \subset \overline{\mathrm D}(0,2)$ but $P_{M_j} = z^2 + b_j$ and $|b_j|>6$. 

The second part \emph{2.} follows immediately from \eqref{Invariance3} and Definition \ref{GmkHmkDef} of the sets $G^k_m$, $H^k_m$ and this then finishes the proof of the lemma. 
\end{proof}

In order to control the geometry of our iterated Julia sets, we will be concerned with the size of the preimages $G^k_{M_k - 1}$, $H^k_{M_k - 1}$ of $\overline {\mathrm D}(0, 2)$, in particular the maximum and minimum distances of  $\partial G^k_{M_k - 1}$, $\partial H^k_{M_k - 1}$ from each of the corresponding points $\sqrt{-b_k}$, $-\sqrt{-b_k}$ respectively. First we prove the following lemma which will be of use to us a number of times.

\begin{lemma}
\label{ForwardImage}
If $b < 0$ and $P(z) = z^2 + b$, then for any $R >0$
$$P({\mathrm D}(\pm \sqrt {-b}, \sqrt{-b + R} - \sqrt{-b})) \subset {\mathrm D}(0, R).$$ 
\end{lemma}

\vspace{-.4cm}
\begin{proof}
Let $r = \sqrt{-b + R} - \sqrt{-b}$ and let $z = \pm\sqrt{-b} + se^{i \theta} \in \overline {\mathrm D}(\pm\sqrt{-b}, r)$. Then 
\begin{eqnarray*}
|P(z)| &=& |z^2 + b|\\
&=& |(\pm\sqrt{-b} + se^{i \theta})^2 + b|\\
&=& |\pm2\sqrt{-b}\,se^{i \theta} + s^2e^{2 i \theta}|\\
&=& |\pm2\sqrt{-b}\,s + s^2e^{i \theta}|.
\end{eqnarray*}
The maximum value for this is clearly attained when $s = r$ and $\theta = 0$ or $\pi$ as appropriate and one checks easily using simple algebra that this maximum value is $2\sqrt{-b}\,r + r^2 = R$ as desired. 
\end{proof}

\vspace{.2cm}
\begin{lemma}
\label{rksk}
For each $k \ge 1$ we have the following: 

\begin{enumerate}
\item The maximum distance of $\partial G^k_{M_k - 1}$, $\partial H^k_{M_k - 1}$ from each of the points $\sqrt{-b_k}$, $-\sqrt{-b_k}$ respectively is $r_k : = \sqrt{-b_k} - \sqrt{-b_k - 2}$, and

\vspace{.2cm}
\item The minimum distance of $\partial G^k_{M_k - 1}$, $\partial H^k_{M_k - 1}$ from each of the points $\sqrt{-b_k}$, $-\sqrt{-b_k}$ respectively is $s_k := \sqrt{-b_k+2} - \sqrt{-b_k}$. 

\end{enumerate}

\end{lemma}

\begin{proof}
 For the upper bound, since $|b_k| > 4$, the map $P_{M_k}(z)= z^2 + b_k$ has inverse branches $\pm \sqrt{z- b_k}$ defined on ${\mathrm D}(0, 3) \supset \overline {\mathrm D}(0, 2)$. If $z = re^{i \theta} \in \overline {\mathrm D}(0, 2)$, then (choosing the positive branch of the square root, the argument being identical for the other branch)
 \begin{eqnarray*} 
|\sqrt{re^{i \theta} - b_k} - \sqrt{-b_k}| &=& \left \vert  \int_0^{re^{i \theta}} {\frac{\dz}{2\sqrt{z - b_k}}} \right \vert \\
&=&  \left \vert  \int_0^r {\frac{e^{i \theta}\dt}{2\sqrt{te^{i \theta} - b_k} }}\right \vert\\
& \le &  \int_0^r {\frac{\dt}{2\sqrt{|te^{i \theta} - b_k|}}}\\
& \le &  \int_0^2 {\frac{\dt}{2\sqrt{|te^{i \theta} - b_k|}}}\\
& \le &  \int_0^2 {\frac{\dt}{2\sqrt{-b_k - t}}}\\
&=& \sqrt{-b_k} - \sqrt{-b_k - 2} = r_k\\ 
\end{eqnarray*}
while this maximum value is clearly attained for the point $z = -2$. 

For the lower bound, we look instead at the forward iterate. Taking $s_k =  \sqrt{-b_k+2} - \sqrt{-b_k}$ from the statement above, by Lemma \ref{ForwardImage} $P_{M_k}({\mathrm D}(\pm\sqrt{-b_k}, s_k)) \subset  \overline {\mathrm D}(0, 2)$ so that $P_{M_k}^{-1}(\overline {\mathrm D}(0, 2)) \supset ({\mathrm D}(-\sqrt{-b_k}, s_k) \cup {\mathrm D}(\sqrt{-b_k}, s_k))$. Thus the minimum distance of $\partial G^k_{M_k - 1}$, $\partial H^k_{M_k - 1}$ from each of the corresponding points $\sqrt{-b_k}$, $-\sqrt{-b_k}$ is at least $s_k$ and one sees this bound is sharp by considering the preimages of the point $z=2$ which finishes the proof of the lemma. 
\end{proof}

We next prove a result giving lower and upper bounds for the quantities $s_k$, $r_k$.

\begin{lemma} 
\label{rkskbounds}
$\tfrac{1}{\sqrt{-b_k}} \le r_k \le \tfrac{1}{\sqrt{-b_k-2}} < \tfrac{1}{2}$ and $\tfrac{1}{2\sqrt{-b_k}} \le s_k$ for each $k \ge 1$.
\end{lemma}

\begin{proof}
Applying the mean value theorem for $f(x) = \sqrt{x}$ on 
the interval $[-b_k - 2, -b_k]$ gives a lower bound of $\tfrac{1}{\sqrt{-b_k}}$ and an upper bound of $\tfrac{1}{\sqrt{-b_k - 2}}$. Since $-b_k > 6$, we see that $\sqrt{-b_k-2} > \sqrt{6-2}=2$, and so the second upper bound follows. In a similar way, a lower bound of $\tfrac{1}{\sqrt{-b_k +2}}$ for $s_k$ follows by applying the mean value theorem for the same function to the interval 
$[-b_k, -b_k+2]$. Since $b_k < -6 < -\tfrac{2}{3}$, one checks that in fact $\tfrac{1}{2\sqrt{-b_k}} \le \tfrac{1}{\sqrt{-b_k + 2}}$ whence the result follows. 
\end{proof}

We now introduce the sets $\Hmt$ which will be important from now on as they are closely related to the sets $\Hm$ which appear in the statements of Theorems \ref{MainTh1}.

\begin{definition}
\label{HmtDef}
For each $m \ge 0$, recall the sets $\H_m^k$ from Definition \ref{GmkHmkDef}. We now define the sets 

\[\Hmt : =  
\cap_{k \ge 1}H_m^k = Q_m(\cap_{k \ge 1}H_0^k)\vspace{.3cm}\]
where the equality follows from both parts of Lemma \ref{GmkHmkDecr} which implies that the sets $H_m^k$ are forward invariant and also compact and nested.
\end{definition} 
 
The sets $\Hmt$ possess the following additional invariance properties. 

\vspace{.2cm}
\begin{lemma}
\label{HmtInvariance1}
For each $m \ge 0$ we have the following:

\begin{enumerate}
\item $\Hmt$ is compact, non-empty, and $\Hmt = \cap_{k \ge k_0}H_m^k$ for any $k_0 \in \N$,

\vspace{.2cm}
\item $P_{m+1} (\Hmt) =  \tilde {\mathcal H}_{m+1}$ (i.e. the sets $\Hmt$ form a forward invariant collection),

\item 

\[\Hmt  =  \left \{ \begin{array}{r@{\: ,\quad}l}
P_{m+1}^{-1}(\tilde {\mathcal H}_{m+1}) & \mbox{if} \:\; m \ne M_k - 1 \quad \mbox{for any} \:\; k \ge 1
\vspace{.2cm}\\
P_{m+1}^{-1}(\tilde {\mathcal H}_{m+1}) \cap H^k_{M_k -1}& \mbox{if} \:\; m=M_k - 1\quad \mbox{for some}  \:\;k \ge 1. \end{array} \right .  \]
\end{enumerate}
\end{lemma}

\begin{proof}
\emph{1.} follows immediately from the fact that the sets $H^k_m$ are compact, non-empty, and nested in view of Part \emph{1.} of Lemma \ref{GmkHmkDecr} while \emph{2.} follows immediately from $\Hmt = Q_m(\cap_{k \ge 1}H_0^k)$ in Definition \ref{HmtDef} above. 
Lastly \emph{3.} follows from the definition of the sets $H^k_m$ (Definition \ref{GmkHmkDef}) where one pays special attention at the times $M_j - 1$ for each $j < k$ and also the fact that one can swap an inverse image and an arbitrary intersection of sets. \end{proof}

Finally, we have a result which will enable us to apply B\"uger's result (Theorem \ref{SelfSimilarity}) to deduce that the sets $\Gm$ and $\Hm$ in the statement of Theorem \ref{MainTh1} are dense.

\begin{lemma}
\label{JSubseq}
Let $\Pm$ be a sequence as above which satisfies the two conditions \eqref{Invariance1}, \eqref{Invariance2}. For each $k \ge 1$, let ${\mathcal J}_k'$ be the iterated Julia set for the sequence of compositions $\{Q_{M_{k-1}, M_k}\}_{k=1}^\infty$. Then, for each $k \ge 1$ we have that ${\mathcal J}_k' = \J_{M_k}$, where $ \J_{M_k}$ is the iterated Julia set at time $M_k$ for the sequence $\Pm$. 
\end{lemma}

\begin{proof}
It follows immediately from the definition of iterated Julia set that ${\mathcal J}_k' \subset \J_{M_k}$ for each $k$, so all we need to do is establish containment in the opposite direction. So let $k \ge 1$ be fixed and pick $z \in  \J_{M_k}$. By Theorem \ref{ThmJm}, $z \in \partial \bigcap_{j \ge 1} Q_{M_k}({\mathcal S}_j) \subset \bigcap_{j \ge 1} Q_{M_k}({\mathcal S}_j)$ so that $|Q_{M_k, M_j}(z)| \le 2$ for each $j \ge k$. On the other hand (as remarked in the proof of Theorem \ref{ThmJm}), if $N$ is any neighbourhood of $z$, then $N$ contains a point $w \in \chat \setminus \bigcap_{j \ge 1} Q_{M_k}({\mathcal S}_j)$. Using \eqref{Invariance3} and the definition \eqref{Sk} the sets ${\mathcal S}_j$, these sets are nested and so we can assume without loss of generality that $w \in  \chat \setminus Q_{M_k}({\mathcal S}_{j_0})$ for some $j_0 > k$. Again by \eqref{Sk}, since $j_0 > k$, $Q_{M_k}({\mathcal S}_{j_0}) = Q_{M_k, M_{j_0}}^{-1}(\overline {\mathrm D}(0,2))$ so that 
$|Q_{M_k, M_{j_0}}(w)| > 2$ from which it then follows using \eqref{Invariance1}, \eqref{Invariance2} that $Q_{M_k, M_j}(w) \to \infty$ as $j \to \infty$. Since $N$ was arbitrary, it then follows that $z \in {\mathcal J}_k'$ as desired, which completes the proof. 
\end{proof}


\section{Proof of Theorem \ref{MainTh1}}

Let the numbers $b_k$, $k \ge 1$ and the sets $G^k_m$, $H^k_m$, $k \ge 1$, $m \ge 0$ be defined as above in Definition \ref{GmkHmkDef} and the paragraph immediately preceding this definition. For each $k \ge 1$ we then define $a_k = -\sqrt{-b_k}$. Note that for each $k \ge 1$ we then have $Q_{M_{k-1}, M_k}(0) = P_{M_k}(a_k) = 0$ so that $Q_{M_{k-1}, M_k}$ fixes $0$. Note also that $a_k$ coincides with the preimage of $0$ under $Q_{M_k-1, M_k} = P_{M_k} = z^2 + b_k$ which lies in $H^k_{M_k - 1}$ and is thus relatively far from  $G^k_{M_k - 1}$. Essentially, it is this asymmetry in the placement of $a_k$ with respect to these two sets which is the key idea in the proof of this Theorem (please see Figures \ref{MainTh1Picture} and \ref{HmtPicture} for an illustration of the basic idea of how this asymmetry gives rise to the two very different parts $\Gm$, $\Hm$ of the Julia set in the statement). Given our placement of $a_k$, it is easy to prove the following:

\begin{claim}
\label{HmkConn}
For each $k \ge 1$, $m \ge 0$, the set $H^k_m$ is path connected and thus connected.
\end{claim}

\begin{claimproof} (of Claim \ref{HmkConn})
For each $M_{k-1} \le m \le M_k -1$, the single critical value $a_k$ of $Q_{m, M_k-1}(z) = z^{2^{M_k - m -1}} + a_k$ coincides with the preimage $-\sqrt{-b_k}$ of $0$ under $P_{M_k}$. It follows using locally defined inverse branches that if we let $X$ be a path connected set at time $M_k -1$ which contains $a_k$, such as $H^k_{M_k -1}$ (see Definition \ref{GmkHmkDef}), then the preimage $Q_{m, M_k-1}^{-1}(X)$ will again be a path connected set (at time $m$) which contains $0$. 

If we now consider the path connected set $H^k_{M_k -1}$, it then follows from the above and Definition \ref{GmkHmkDef} that $H_{M_{k-1}}^k$ is a path connected set which, by this definition and \eqref{Invariance3}, must be a subset of $H_{M_{k-1}}^k = \overline{\mathrm D}(0,2)$ and which contains $0$ which is the image under $P_{M_{k-1}}$ of the critical value $a_{k-1}$ of $P_{M_{k-1}-1}$.  Taking a further preimage under $P_{M_{k-1}}(z) = z + b_{k-1}$
and recalling Definition \ref{GmkHmkDef} and the fact that $|b_{k-1}| > 6$, we obtain that $H_{M_{k-1}}^k$ is a path connected set which contains the critical value $a_{k-1}$ of $P_{M_{k-1}-1}$ and thus of $Q_{M_{k-2}, M_{k-1}-1}$.

It then follows by an easy induction starting from the path connected set $H^k_{M_k -1}$ at time $M_k - 1$ and working backwards to time $0$ that $H^k_m$ is path connected for each $0 \le m \le M_k-1$. On the other hand, the path connectedness of $H^k_m$ for $m > M_k -1$ follows immediately from the fact that, for $m > M_k -1$, $H^k_m = Q_{M_k -1, m}(H^k_{M_k -1})$ combined with the path connectedness of $H^k_{M_k -1}$ and with this the proof of the Claim is complete.
\end{claimproof}

To complete the definition of our sequence of polynomials, choose the sequence of positive integers $\{m_k\}_{k=1}^\infty$ to satisfy the two invariance conditions \eqref{Invariance1}, \eqref{Invariance2}  (we observe that these bounds on $m_k$ will now depend entirely on the absolute values of the quantities $b_k$ since $a_k$ is given in terms of $b_k$). In addition, we will need to choose $\{m_k\}_{k=1}^\infty$ to guarantee the existence (at the times $M_k$) of a backward invariant (closed) annulus which will contain the Julia sets $\J_{M_k}$ and also a forward invariant disc which will guarantee the existence of bounded Fatou components. 

\afterpage{\clearpage}

\begin{figure}
\vspace{.5cm}
\begin{tikzpicture}

\node at (-5.1,11.5) {\footnotesize Stage $M_{k_0+2}$};

\node at (1.6, 9.1) {$\scriptstyle P_{M_{k_0+2}}$};

{\color{darkred}
\draw[line width=0.1mm, fill=evenlesspalepink] (.8,11.5) circle (1.25);}

 {\color{black}
  \draw[line width=0.15mm] (.75,11.45) -- (.85,11.55);
  \draw[line width=0.15mm] (.75,11.55) -- (.855,11.45);
  
  \node at (1.02,10.85) {$\scriptstyle 0$};
  
  }
  
  \draw[line width=0.2mm, ->] (.8,8.4) -- (.8,9.9);

\node at (3.4, 11.4) {\color{darkred}$\scriptstyle H^{k_0+2}_{M_{k_0+2}} = \overline {\mathrm{D}}(0,2)$};

\node at (-4.8,8.0) {\footnotesize Stage $M_{k_0+2}-1$};

  \node at (3.2, 7.95) {\color{darkblue}$\scriptstyle  G^{k_0+2}_{M_{{k_0+2}}-1}$};

  {\color{darkblue}
  \draw[fill=evenlesspaleblue] (2,8) circle (.3);}
  
  \node at (-1.6, 7.95) {\color{darkred}$\scriptstyle  H^{k_0+2}_{M_{{k_0+2}}-1}$};

   {\color{darkred}
  \draw[fill=evenlesspalepink] (-.4,8) circle (.3);}
  
  {\color{black}
  \draw[line width=0.15mm] (-.45,7.95) -- (-.35,8.05);
  \draw[line width=0.15mm] (-.45,8.05) -- (-.35,7.95);
  
  \node at (-.6,6.9) {$\scriptscriptstyle a_{k_0+2} = - \sqrt{-b_{k_0+2}}$};
  
  \draw[line width=0.15mm, ->] (-.4,7.2) -- (-.4,7.9);
  }
  
     {\color{black}
  \draw[line width=0.15mm] (.75,7.95) -- (.85,8.05);
  \draw[line width=0.15mm] (.75,8.05) -- (.855,7.95);
  
  \node at (1.02,7.85) {$\scriptstyle 0$};
  
  }

\draw[line width=0.2mm, ->] (.8,4) -- (.8,7.6);

\node at (2.3, 5.7) {$\scriptstyle Q_{M_{k_0+1},M_{k_0+2} -1}$};

   \node at (-4.9,2.4) {\footnotesize Stage $M_{k_0+1}$};
  
  {\color{darkred}
\draw[line width=0.1mm, fill=lesspalepink] (.8,2.4) circle (1.25);}

{\color{darkred}
\draw[line width=0.1mm, fill=evenlesspalepink] (.8,2.4) circle (0.833);}

{\color{darkblue}
  \foreach \phi in {0,22.5,...,360}{
    \draw[line width=0.1mm,fill=evenlesspaleblue] ({.8+1.04*cos(\phi)},{2.4+1.04*sin(\phi)}) circle (.06);
  }}
  
     {\color{black}
  \draw[line width=0.15mm] (.75,2.35) -- (.85,2.45);
  \draw[line width=0.15mm] (.75,2.45) -- (.855,2.35);
  
  \node at (1.02,2.25) {$\scriptstyle 0$};
  
  }
  
  {\color{darkred}
  
    \node at (-1.2, 3.9) {\color{darkred}$\scriptstyle  H^{k_0+2}_{M_{k_0+1}}$};
  
    \draw[line width=0.15mm, ->] (-.8,3.55) -- (0.05,2.9);
  }

 {\color{darkblue}
  
    \node at (2.6, 3.9) {\color{darkblue}$\scriptstyle  G^{k_0+2}_{M_{k_0+1}}$};
  
    \draw[line width=0.15mm, ->] (2.12,3.75) -- (1.6,3.2);
  }

\draw[line width=0.2mm, ->] (.8,-.3) -- (.8,0.8);  

\node at (1.6, 0.1) {$\scriptstyle P_{M_{k_0+1}}$};
  
    {\color{darkred}
  \draw[line width=0.05mm,, fill=lesspalepink] (-.4,-.6) circle (.3);}
  
      {\color{darkred}
  \draw[line width=0.05mm,, fill=evenlesspalepink] (-.4,-.6) circle (.2);
  
  \node at (-1.7, -.6) {\color{darkred}$\scriptstyle  H^{k_0+1}_{M_{k_0+1}-1}$};
  
  \node at (-.7, 0.7) {\color{darkred}$\scriptstyle  H^{k_0+2}_{M_{k_0+1}-1}$};
  
    \draw[line width=0.15mm, ->] (-.7,0.4) -- (-.48,-.4);
  }
  
  {\color{black}
  \draw[line width=0.15mm] (-.45,-.65) -- (-.35,-.55);
  \draw[line width=0.15mm] (-.45,-.55) -- (-.35,-.65);
  
  \node at (-.6,-1.8) {$\scriptscriptstyle a_{k_0+1} = - \sqrt{-b_{k_0+1}}$};
  
  \draw[line width=0.15mm, ->] (-.4,-1.5) -- (-.4,-.7);
  }
  
   \node at (3.7, 2.4) {\color{darkred} \footnotesize$ H^{k_0+1}_{M_{k_0+1}} = \overline {\mathrm{D}}(0,2)$};
  
     \node at (-4.6,-.6) {\footnotesize Stage $M_{k_0+1} - 1$};
  
   {\color{darkblue}
  \draw[line width=0.05mm,, fill=lesspaleblue] (2,-.6) circle (.3);}
  
  {\color{darkblue}
  \foreach \phi in {0,22.5,...,360}{
    \draw[line width=0.05mm,fill=evenlesspaleblue] ({2+.24*cos(\phi)},{-.6+.25*sin(\phi)}) circle (.02);
  }}

   \node at (3.3, -.6) {\color{darkblue}$\scriptstyle  G^{k_0+1}_{M_{k_0+1}-1}$};

    {\color{darkblue}
  
    \node at (3.8, 0.7) {\color{darkblue}$\scriptstyle  G^{k_0+2}_{M_{k_0+1}-1}$};
  
    \draw[line width=0.15mm, ->] (3.2,0.6) -- (2.2,-.4);
  }
  
    {\color{black}
  \draw[line width=0.15mm] (.75,-.55) -- (.85,-.65);
  \draw[line width=0.15mm] (.75,-.65) -- (.85,-.55);
  
  \node at (1.03,-.702) {$\scriptstyle 0$};
 
  }   
  
  \draw[line width=0.2mm, ->] (.8,-4.6) -- (.8,-0.9);
  
  \node at (2.1, -2.8) {$\scriptstyle Q_{M_{k_0},M_{k_0+1} -1}$};
  
     \node at (3.5, -6.2) {\color{darkred} \footnotesize$ H^{k_0}_{M_{k_0}} = \overline {\mathrm{D}}(0,2)$};
  
    \node at (-4.9,-6.2) {\footnotesize Stage $M_{k_0}$};
  
    {\color{darkred}
\draw[line width=0.1mm, fill=palepink] (.8,-6.2) circle (1.25);}

{\color{darkred}
\draw[line width=0.1mm, fill=lesspalepink] (.8,-6.2) circle (0.8);}

{\color{darkred}
\draw[line width=0.1mm, fill=evenlesspalepink] (.8,-6.2) circle (0.72);}

   {\color{black}
  \draw[line width=0.15mm] (.75,-6.15) -- (.85,-6.25);
  \draw[line width=0.15mm] (.75,-6.25) -- (.855,-6.15);
  
  \node at (1.02,-6.35) {$\scriptstyle 0$};
  
  }

{\color{darkblue}
  \foreach \phi in {0,45,...,360}{
    \draw[line width=0.1mm,fill=lesspaleblue] ({.8+1.025*cos(\phi)},{-6.2+1.025*sin(\phi)}) circle (.12);
  }}

{\color{darkblue}
\foreach \psi in {0, 45,..., 360}{
  \foreach \phi in {0,22.5,...,360}{
    \draw[line width=0.02mm,fill=evenlesspaleblue] ({.8+1.025*cos(\psi) + .1*cos(\phi)},{-6.2+1.025*sin(\psi) + .1*sin(\phi)}) circle (.01);
  }}}

   {\color{darkred}
  
    \node at (-2.2, -5.05) {\color{darkred}$\scriptstyle  H^{k_0+1}_{M_{k_0}}, \,\,H^{k_0+2}_{M_{k_0}}$};
  
    \draw[line width=0.15mm, ->] (-1.25,-5.35) -- (0.0,-5.87);
  }

    {\color{darkblue}
  
    \node at (3.4, -4.2) {\color{darkblue}$\scriptstyle  G^{k_0+1}_{M_{k_0}}, \,\,G^{k_0+2}_{M_{k_0}}$};
  
    \draw[line width=0.15mm, ->] (2.6,-4.4) -- (1.63,-5.37);
  }


\end{tikzpicture}   
\caption{The setup for Theorem \ref{MainTh1}} \label{MainTh1Picture}
\end{figure}
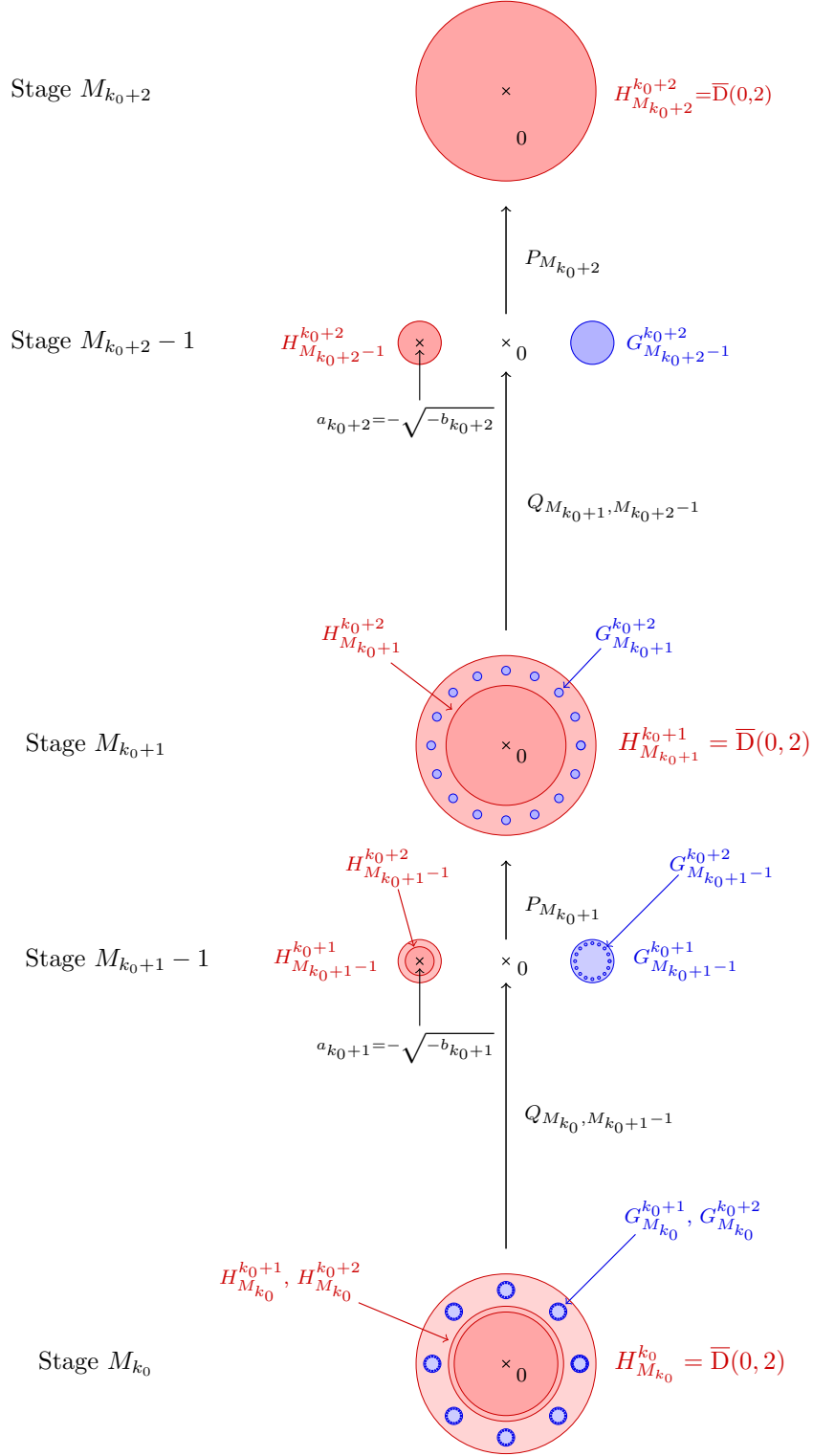

\newpage

Together this allows us to be sure we have a bounded Fatou component whose closure will be one of the uniformly perfect pieces of the iterated filled Julia set in part \emph{5.} of the statement.

\begin{claim}
\label{InvariantAnnuliDiscs}
If in addition to satisfying \eqref{Invariance1}, \eqref{Invariance2} the integers $m_k$, $k \ge 1$ also satisfy
\begin{equation}
\label{mkbound1}
m_k \ge \frac{\log \! \left (\frac{\log{8\sqrt{-b_k}}}{\log2} \right)}{\log 2},\vspace{.2cm}
\end{equation}
then we can ensure that 

\begin{enumerate}

\item 
$$Q_{M_{k-1}, M_k}^{-1}(\overline {\mathrm A}(0, \tfrac{1}{2}, 2)) \subset \overline {\mathrm A}(0, \tfrac{1}{2}, 2),$$

\item
$$Q_{M_{k-1}, M_k}^{-1}({\mathrm D}(0, \tfrac{1}{2})) \supset {\mathrm D}(0, \tfrac{1}{2}),$$

\item
$$H^k_{M_{k-1}} = Q_{M_{k-1}, M_k-1}^{-1}(H^k_{M_k-1}) \supset {\mathrm D}(0, \tfrac{1}{2}).$$
\end{enumerate}
\end{claim}

\begin{claimproof} (of Claim \ref{InvariantAnnuliDiscs})
We begin by defining $t_k := \sqrt{-b_k + \tfrac{1}{2}} - \sqrt{-b_k}$ for each $k \ge 1$. By Lemma \ref{ForwardImage} $P_{M_k}({\mathrm D}(\pm \sqrt{-b_k}, t_k)) \subset {\mathrm D}(0, \tfrac{1}{2})$ so that 
$$P_{M_k}^{-1} ({\mathrm D}(0, \tfrac{1}{2})) \supset \left ( {\mathrm D}(-\sqrt{-b_k}, t_k) \cup {\mathrm D}(\sqrt{-b_k}, t_k) \right ).$$

Since $b_k < -6$, using either the mean value theorem for the function $f(x) = \sqrt x$ on the interval $[\sqrt{-b_k}, \sqrt{-b_k + \tfrac{1}{2}}]$ or just simple algebra, one obtains that $t_k > \tfrac{1}{8\sqrt{-b_k}}$, so that from above we also have 
\begin{equation}
\label{InverseDisc1}
P_{M_k}^{-1}({\mathrm D}(0, \tfrac{1}{2})) \supset ({\mathrm D}(-\sqrt{-b_k}, \tfrac{1}{8\sqrt{-b_k}}) \cup {\mathrm D}(\sqrt{-b_k}, \tfrac{1}{8\sqrt{-b_k}}))\vspace{.1cm}
\end{equation}
whence by Lemma \ref{rksk}
\vspace{.1cm}
\begin{equation}
\label{InverseAnnulus1}
P_{M_k}^{-1}(\overline {\mathrm A}(0, \tfrac{1}{2}, 2)) \subset \left ( \overline{\mathrm A}(-\sqrt{-b_k}, \tfrac{1}{8\sqrt{-b_k}},r_k) \cup \overline{\mathrm A}(\sqrt{-b_k}, \tfrac{1}{8\sqrt{-b_k}}, r_k) \right ). 
\end{equation}

If $m_k$ is bounded below as in \eqref{mkbound1}, then, again since $b_k < -6$,  we have 

\begin{equation}
\label{OuterRadiusBound1}
2^{2^{m_k}} \ge 8\sqrt{-b_k} > 2
\end{equation}

whence 
\begin{equation}
\label{InnerRadiusBound1}
\left ( \frac{1}{8\sqrt{-b_k}}\right)^{1/2^{m_k}} \ge \frac{1}{2}.
\end{equation}

Recalling from our construction using \eqref{PmkDef} of the sequence $\Pm$ that $Q_{M_{k-1},M_k -1}(z) = z^{2^{m_k}} + a_k = z^{2^{m_k}} - \sqrt{-b_k}$ as well as our choice of $a_k = - \sqrt{-b_k}$ from the start of the proof (of Theorem \ref{MainTh1}), it then follows using  \eqref{InverseAnnulus1}, \eqref{OuterRadiusBound1}, \eqref{InnerRadiusBound1} that (using $r_k < \tfrac{1}{2}$ from Lemma \ref{rkskbounds} so that $r_k <  \tfrac{1}{2} < +  \sqrt 6 < \sqrt{-b_k}$ and thus $2\sqrt{-b_k} +  r_k <  3\sqrt{-b_k} < 8 \sqrt{-b_k}$)
\begin{eqnarray*}
\hspace{0cm}Q_{M_{k-1}, M_k}^{-1}(\overline {\mathrm A}(0, \tfrac{1}{2}, 2))&\\ & \hspace{-2cm}\subset Q_{M_{k-1}, M_k-1}^{-1}\left (\overline{\mathrm A}(-\sqrt{-b_k}, \tfrac{1}{8\sqrt{-b_k}},r_k) \cup \overline{\mathrm A}(\sqrt{-b_k}, \tfrac{1}{8\sqrt{-b_k}}, r_k)\right )&\\
&\hspace{-5cm} \subset Q_{M_{k-1}, M_k-1}^{-1}\left (\overline{\mathrm A}(-\sqrt{-b_k}, \tfrac{1}{8\sqrt{-b_k}}, 8\sqrt{-b_k})\right )&\\ 
&\hspace{-10.55cm}  \subset \overline {\mathrm A}(0, \tfrac{1}{2}, 2)&
\end{eqnarray*}
from which we obtain \emph{1.} On the other hand, using \eqref{InverseDisc1}, \eqref{InnerRadiusBound1}
\begin{eqnarray*}
Q_{M_{k-1}, M_k}^{-1}({\mathrm D}(0, \tfrac{1}{2}))\!\!\! &=& 
Q_{M_{k-1}, M_k-1}^{-1}(P_{M_k}^{-1}({\mathrm D}(0, \tfrac{1}{2}))\\ 
&\supset&   Q_{M_{k-1}, M_k-1}^{-1}({\mathrm D}(-\sqrt{-b_k}, \tfrac{1}{8\sqrt{-b_k}}) \cup {\mathrm D}(\sqrt{-b_k}, \tfrac{1}{8\sqrt{-b_k}}))\\
&\supset&  Q_{M_{k-1}, M_k-1}^{-1}({\mathrm D}(-\sqrt{-b_k}, \tfrac{1}{8\sqrt{-b_k}}))\\
&\supset&  {\mathrm D}(0, \tfrac{1}{2})
\end{eqnarray*}
from which we obtain \emph{2.} 

Finally, using Definition  \ref{GmkHmkDef} together with Lemmas \ref{rksk}, \ref{rkskbounds} and \eqref{InnerRadiusBound1}
\begin{eqnarray*} 
Q_{M_{k-1}, M_k-1}^{-1}(H^k_{M_k-1}) \!\!\!&\supset& Q_{M_{k-1}, M_k-1}^{-1}({\mathrm D}(-\sqrt{-b_k},s_k))\\
&\supset& Q_{M_{k-1}, M_k-1}^{-1}({\mathrm D}(-\sqrt{-b_k}, \tfrac{1}{2\sqrt{-b_k}}))\\
&\supset& Q_{M_{k-1}, M_k-1}^{-1}({\mathrm D}(-\sqrt{-b_k}, \tfrac{1}{8\sqrt{-b_k}}))\\
&\supset& {\mathrm D}(0, \tfrac{1}{2})
\end{eqnarray*}
while $H^k_{M_{k-1}} = Q_{M_{k-1}, M_k-1}^{-1}(H^k_{M_k-1})$ by definition of the set $H^k_{M_{k-1}}$ (Definition  \ref{GmkHmkDef}). With this we have
\emph{3.} and the proof of the claim is complete. \end{claimproof}

From now on we assume the integers $m_k$ satisfy the lower bound \eqref{mkbound1} above (we shall be replacing this with a stricter lower bound \eqref{mkbound2} later on). We will require an upper bound on the derivatives of the inverse branches of the maps $Q_{M_{k-1},M_k}$ on an open annulus which contains $\overline {\mathrm A}(0, \tfrac{1}{2}, 2)$. Our first step in this direction is the following:

\begin{claim}
\label{InverseBranchLowerBound1}
For $z \in P_{M_k}^{-1}({\mathrm A}(0,\tfrac{1}{3}, 3))$, we have 

$$|z - a_k| > \frac{1}{12\sqrt{-b_k}}.$$
\end{claim}

\begin{claimproof}  (of Claim \ref{InverseBranchLowerBound1})
By Lemma \ref{ForwardImage} (and the continuity of $P_{M_k}$), if we define $u_k := \sqrt{-b_k + \tfrac{1}{3}} - \sqrt{-b_k}$, then $P_{M_k}(\overline {\mathrm D}(\pm \sqrt{-b_k}, u_k)) \subset \overline {\mathrm D}(0, \tfrac{1}{3})$ whence we have 
$$P_{M_k}^{-1}({\mathrm A}(0,\tfrac{1}{3}, 3)) \cap (\overline {\mathrm D}(-\sqrt{-b_k}, u_k) \cup \overline {\mathrm D}(\sqrt{-b_k}, u_k)) = \emptyset$$
so that, if $z  \in P_{M_k}^{-1}({\mathrm A}(0,\tfrac{1}{3}, 3))$, then $|z - a_k| > u_k$ (recall that here we have $a_k = -\sqrt{-b_k}$).

A very similar argument to that for the lower bound for $s_k$ in Lemma \ref{rkskbounds} using the mean value theorem and the fact $|b_k| > 6$ then shows that $u_k > \tfrac{1}{12\sqrt{-b_k}}$ from which the claim then follows. 
\end{claimproof}

The above claim is the key to proving the desired estimates on the derivatives of the inverse branches of the maps $Q_{M_{k-1},M_k}$ on
${\mathrm A}(0,\tfrac{1}{3}, 3)$. For each $k \ge 1$ each of the polynomials $Q_{m,M_k}$, $M_{k-1} \le m \le M_k -2$ has two critical values: $0$ which is the image of the single critical value $a_k$ of $Q_{m, M_k -1}$ under $P_{M_k}$ and the critical value $b_k$ of $P_{M_k}$. Since $|b_k| > 6$ both these critical values avoid the annulus ${\mathrm A}(0,\tfrac{1}{3}, 3)$. 

Again since $|b_k| > 6$, the preimage of this annulus under $P_{M_k}$ consists of two annuli (of the same modulus as the original), and exactly one of these contains the critical value $a_k = -\sqrt{-b_k}$ of $Q_{M_{k-1}, M_k -1}$ in its bounded component. Away from this critical value, the preimage of a point $z$ under $Q_{M_{k-1}, M_k -1}$ consists of $2^{m_k}$ points which are rotations of each other by an angle of $\tfrac{2\pi}{2^{m_k}}$ and one runs through all of these preimages by analytically continuing 
$Q_{M_{k-1}, M_k -1}^{-1}$ by running $2^{m_k}$ times around a simple closed curve enclosing $a_k$ which starts and ends at $z$.

Putting this together, we see that on ${\mathrm A}(0,\tfrac{1}{3}, 3)$ we have one analytic continuation of $Q_{M_{k-1},M_k}^{-1}$ obtained by going around this annulus $2^{m_k}$ times which corresponds to choosing the component of $P_{M_k}^{-1}({\mathrm A}(0,\tfrac{1}{3}, 3))$ which contains $0$ in its bounded complementary component. In addition, we have a further $2^{m_k}$ components of the preimage which correspond to branches of $Q_{M_{k-1},M_k}^{-1}$ where we choose a component of $P_{M_k}^{-1}({\mathrm A}(0,\tfrac{1}{3}, 3))$ which does not contain $0$ in its bounded complementary component and where we have a single branch of $Q_{M_{k-1},M_k}^{-1}$ which is defined and univalent on the disc ${\mathrm D}(0, 3)$. Thus, for any locally defined branch $f$ of $Q_{M_{k-1},M_k}^{-1}$ on ${\mathrm A}(0,\tfrac{1}{3}, 3)$, $f$ can either be analytically continued by going around this annulus $2^{m_k}$ times or can be extended to be globally defined and univalent on the disc ${\mathrm D}(0, 3)$ which contains this annulus. 

\newpage
\begin{claim}
\label{Contracting1}
For each $k \ge 1$, we have the following: 

\begin{enumerate}
\item For each $M_{k-1} \le m \le M_k - 1$ and each  locally defined branch $f$ of $Q_{m, M_k}^{-1}$ defined on ${\mathrm A}(0,\tfrac{1}{3}, 3)$, we have 

$$|f'(z)| < 6\sqrt{-b_k}, \quad \mbox{for each} \:\; z \in {\mathrm A}(0,\tfrac{1}{3}, 3),$$

\vspace{.2cm}
\item If in addition to satisfying \eqref{Invariance1}, \eqref{Invariance2}, and \eqref{mkbound1}, we also require that 

\begin{equation}
\label{mkbound2}
m_k \ge \tfrac{\log(12\sqrt{-b_k})}{\log 2},\\
\end{equation}

\vspace{.2cm}
for each locally defined branch $f$ of $Q_{M_{k-1}, M_k}^{-1}$ defined on ${\mathrm A}(0,\tfrac{1}{3}, 3)$, we have 

$$|f'(z)| < \frac{1}{2}, \quad \mbox{for each} \:\; z \in {\mathrm A}(0,\tfrac{1}{3}, 3).\hspace{.4cm}$$

\end{enumerate}
\end{claim}

We draw the reader's attention to the fact that we will need both parts of the above to prove Claim \ref{ShrinkingSeparatingAnnuli1} when we show that the size of the annuli $B^{k,j}_m$ tends to zero as $k$ tends to infinity.

\begin{claimproof}  (of Claim \ref{Contracting1})
Since $|b_k| > 6 > 4$, if $z \in {\mathrm D}(0,3) \supset {\mathrm A}(0,\tfrac{1}{3}, 3)$, then $|z - b_k| > 1$. Using this, the chain rule, and the previous claim (Claim \ref{InverseBranchLowerBound1}),  if $z \in {\mathrm A}(0,\tfrac{1}{3}, 3)$ and $M_{k-1} \le m \le M_k - 1$,
\begin{eqnarray*}
|f'(z)| &=& \frac{1}{2^{M_k-m-1}|\sqrt{z - b_k} - a_k|^{1-\tfrac{1}{2^{M_k-m-1}}}}\cdot \frac{1}{2\sqrt{|z - b_k|}}\\
&<& \frac{1}{2^{M_k-m}}\cdot 12\sqrt{-b_k}\\
\end{eqnarray*}
from which \emph{1.} follows immediately. \emph{2.} follows on setting $m = M_{k-1}$ and making use of the lower bound \eqref{mkbound2} above and recalling from Definition \ref{Invariance1} that $M_k - M_{k-1} = m_k + 1$
(we leave it as an easy exercise to the reader to verify that this lower bound is larger than the earlier lower bound \eqref{mkbound1} on $m_k$ which was required to ensure the invariance condition \emph{1.} in Claim \ref{InvariantAnnuliDiscs}). \end{claimproof}

From now on we will assume the integers $m_k$ satisfy the lower bound \eqref{mkbound2} above.

Before we can begin verifying the five parts of the statement of Theorem \ref{MainTh1}, we need to establish some more claims. Recall the sets $\Hmt = \cap_{k \ge 1}H_m^k = Q_m(\cap_{k \ge 1}H_0^k)$ introduced in Definition \ref{HmtDef} just before the statement of Lemma \ref{HmtInvariance1}. 
\vspace{.3cm}

\begin{claim}
\label{Hmt}
For each $m \ge 0$ we have the following:
\begin{enumerate}
\item $\interior \,\Hmt$ is a non empty subset of the iterated Fatou set $\Fm$ with $a_k \in \tilde {\mathcal H}_m$ if $m = M_k -1$ for some $k$ while $0 \in \tilde {\mathcal H}_m$ for all other values of $m$,
\vspace{.25cm}
\item $\partial \Hmt \subset \Jm$,
\vspace{.25cm}
\item  $P_{m+1}(\interior \,\Hmt) = \interior\, \tilde {\mathcal H}_{m+1}$ and $P_{m+1}(\partial \Hmt) =\partial \tilde {\mathcal H}_{m+1}$,
\vspace{.25cm}
\item $\chat \setminus \Hmt$ is simply connected, $\Hmt$ is connected, while $\partial \Hmt$ is also connected and hence uniformly perfect. 
\end{enumerate}
\end{claim}

The reader might find it helpful to consult Figures \ref{MainTh1Picture}, \ref{HmtPicture}, and \ref{Mk0Picture} before embarking on the details of the proof below. It is particularly instructive to examine the difference between the first two figures, the main one being that Figure \ref{MainTh1Picture} shows merely the sets $G^k_m$, $H^k_m$ at certain times while Figure \ref{HmtPicture} shows all the preimage components of the sets $\overline {\mathrm D}(0,2)$ at times $M_k$, the `extra' components being shown in purple.

\afterpage{\clearpage}

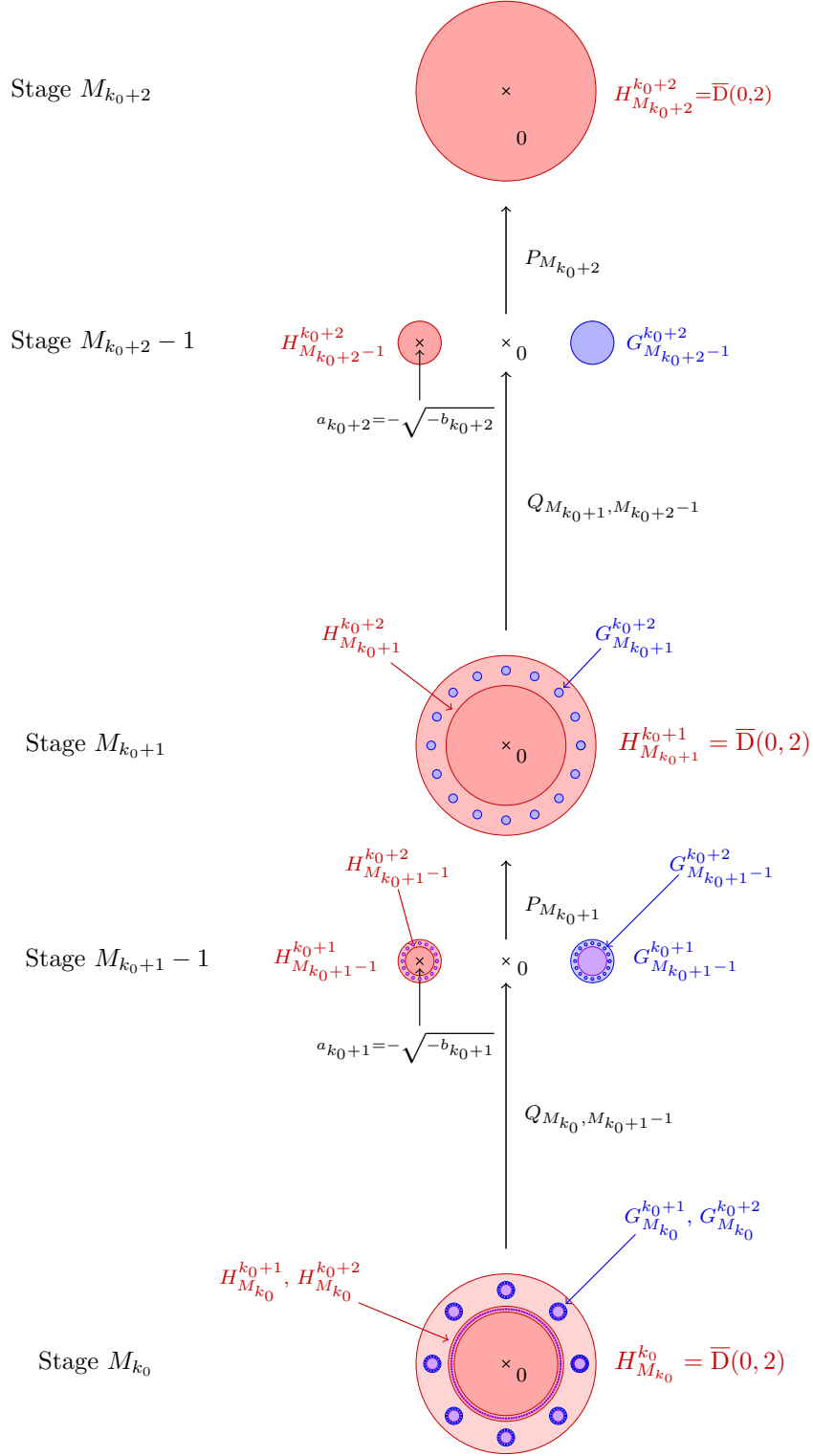
\begin{figure}
\vspace{.5cm}
\begin{tikzpicture}

\node at (-5.1,11.5) {\footnotesize Stage $M_{k_0+2}$};

\node at (1.6, 9.1) {$\scriptstyle P_{M_{k_0+2}}$};

{\color{darkred}
\draw[line width=0.1mm, fill=evenlesspalepink] (.8,11.5) circle (1.25);}

 {\color{black}
  \draw[line width=0.15mm] (.75,11.45) -- (.85,11.55);
  \draw[line width=0.15mm] (.75,11.55) -- (.855,11.45);
  
  \node at (1.02,10.85) {$\scriptstyle 0$};
  
  }
  
  \draw[line width=0.2mm, ->] (.8,8.4) -- (.8,9.9);

\node at (3.4, 11.4) {\color{darkred}$\scriptstyle H^{k_0+2}_{M_{k_0+2}} = \overline {\mathrm{D}}(0,2)$};

\node at (-4.8,8.0) {\footnotesize Stage $M_{k_0+2}-1$};

  \node at (3.2, 7.95) {\color{darkblue}$\scriptstyle  G^{k_0+2}_{M_{{k_0+2}}-1}$};

  {\color{darkblue}
  \draw[fill=evenlesspaleblue] (2,8) circle (.3);}
  
  \node at (-1.6, 7.95) {\color{darkred}$\scriptstyle  H^{k_0+2}_{M_{{k_0+2}}-1}$};

   {\color{darkred}
  \draw[fill=evenlesspalepink] (-.4,8) circle (.3);}
  
  {\color{black}
  \draw[line width=0.15mm] (-.45,7.95) -- (-.35,8.05);
  \draw[line width=0.15mm] (-.45,8.05) -- (-.35,7.95);
  
  \node at (-.6,6.9) {$\scriptscriptstyle a_{k_0+2} = - \sqrt{-b_{k_0+2}}$};
  
  \draw[line width=0.15mm, ->] (-.4,7.2) -- (-.4,7.9);
  }
  
     {\color{black}
  \draw[line width=0.15mm] (.75,7.95) -- (.85,8.05);
  \draw[line width=0.15mm] (.75,8.05) -- (.855,7.95);
  
  \node at (1.02,7.85) {$\scriptstyle 0$};
  
  }

\draw[line width=0.2mm, ->] (.8,4) -- (.8,7.6);

\node at (2.3, 5.7) {$\scriptstyle Q_{M_{k_0+1},M_{k_0+2} -1}$};

   \node at (-4.9,2.4) {\footnotesize Stage $M_{k_0+1}$};
  
  {\color{darkred}
\draw[line width=0.1mm, fill=lesspalepink] (.8,2.4) circle (1.25);}

{\color{darkred}
\draw[line width=0.1mm, fill=evenlesspalepink] (.8,2.4) circle (0.833);}

{\color{darkblue}
  \foreach \phi in {0,22.5,...,360}{
    \draw[line width=0.1mm,fill=evenlesspaleblue] ({.8+1.04*cos(\phi)},{2.4+1.04*sin(\phi)}) circle (.06);
  }}
  
     {\color{black}
  \draw[line width=0.15mm] (.75,2.35) -- (.85,2.45);
  \draw[line width=0.15mm] (.75,2.45) -- (.855,2.35);
  
  \node at (1.02,2.25) {$\scriptstyle 0$};
  
  }
  
  {\color{darkred}
  
    \node at (-1.2, 3.9) {\color{darkred}$\scriptstyle  H^{k_0+2}_{M_{k_0+1}}$};
  
    \draw[line width=0.15mm, ->] (-.8,3.55) -- (0.05,2.9);
  }

 {\color{darkblue}
  
    \node at (2.6, 3.9) {\color{darkblue}$\scriptstyle  G^{k_0+2}_{M_{k_0+1}}$};
  
    \draw[line width=0.15mm, ->] (2.12,3.75) -- (1.6,3.2);
  }

\draw[line width=0.2mm, ->] (.8,-.3) -- (.8,0.8);  

\node at (1.6, 0.1) {$\scriptstyle P_{M_{k_0+1}}$};
  
    {\color{darkred}
  \draw[line width=0.05mm,, fill=lesspalepink] (-.4,-.6) circle (.3);}
  
      {\color{darkred}
  \draw[line width=0.05mm,, fill=evenlesspalepink] (-.4,-.6) circle (.2);
  
  \node at (-1.7, -.6) {\color{darkred}$\scriptstyle  H^{k_0+1}_{M_{k_0+1}-1}$};
  
  \node at (-.7, 0.7) {\color{darkred}$\scriptstyle  H^{k_0+2}_{M_{k_0+1}-1}$};
  
    \draw[line width=0.15mm, ->] (-.7,0.4) -- (-.48,-.4);
  }
  
  {\color{black}
  \draw[line width=0.15mm] (-.45,-.65) -- (-.35,-.55);
  \draw[line width=0.15mm] (-.45,-.55) -- (-.35,-.65);
  
  \node at (-.6,-1.8) {$\scriptscriptstyle a_{k_0+1} = - \sqrt{-b_{k_0+1}}$};
  
  \draw[line width=0.15mm, ->] (-.4,-1.5) -- (-.4,-.7);
  }
  
   \node at (3.7, 2.4) {\color{darkred} \footnotesize$ H^{k_0+1}_{M_{k_0+1}} = \overline {\mathrm{D}}(0,2)$};
  
     \node at (-4.6,-.6) {\footnotesize Stage $M_{k_0+1} - 1$};
  
   {\color{darkblue}
  \draw[line width=0.05mm,, fill=lesspaleblue] (2,-.6) circle (.3);}
  
  {\color{darkblue}
  \foreach \phi in {0,22.5,...,360}{
    \draw[line width=0.05mm,fill=evenlesspaleblue] ({2+.24*cos(\phi)},{-.6+.25*sin(\phi)}) circle (.02);
  }}
  
       {\color{darkpurple}
  \draw[line width=0.05mm,, fill=evenlesspalepurple] (2,-.6) circle (.2);}
 
     \node at (3.3, -.6) {\color{darkblue}$\scriptstyle  G^{k_0+1}_{M_{k_0+1}-1}$};
  
   {\color{darkpurple}
  \foreach \phi in {0,22.5,...,360}{
    \draw[line width=0.05mm,fill=evenlesspalepurple] ({-.4+.24*cos(\phi)},{-.6+.25*sin(\phi)}) circle (.02);
  }}
  
    {\color{darkblue}
  
    \node at (3.8, 0.7) {\color{darkblue}$\scriptstyle  G^{k_0+2}_{M_{k_0+1}-1}$};
  
    \draw[line width=0.15mm, ->] (3.2,0.6) -- (2.2,-.4);
  }
  
    {\color{black}
  \draw[line width=0.15mm] (.75,-.55) -- (.85,-.65);
  \draw[line width=0.15mm] (.75,-.65) -- (.85,-.55);
  
  \node at (1.03,-.702) {$\scriptstyle 0$};
 
  }   
  
  \draw[line width=0.2mm, ->] (.8,-4.6) -- (.8,-0.9);
  
  \node at (2.1, -2.8) {$\scriptstyle Q_{M_{k_0},M_{k_0+1} -1}$};
  
     \node at (3.5, -6.2) {\color{darkred} \footnotesize$ H^{k_0}_{M_{k_0}} = \overline {\mathrm{D}}(0,2)$};
  
    \node at (-4.9,-6.2) {\footnotesize Stage $M_{k_0}$};
  
    {\color{darkred}
\draw[line width=0.1mm, fill=palepink] (.8,-6.2) circle (1.25);}

{\color{darkred}
\draw[line width=0.1mm, fill=lesspalepink] (.8,-6.2) circle (0.8);}

{\color{darkred}
\draw[line width=0.1mm, fill=evenlesspalepink] (.8,-6.2) circle (0.72);}

   {\color{black}
  \draw[line width=0.15mm] (.75,-6.15) -- (.85,-6.25);
  \draw[line width=0.15mm] (.75,-6.25) -- (.855,-6.15);
  
  \node at (1.02,-6.35) {$\scriptstyle 0$};
  
  }

{\color{darkblue}
  \foreach \phi in {0,45,...,360}{
    \draw[line width=0.1mm,fill=lesspaleblue] ({.8+1.025*cos(\phi)},{-6.2+1.025*sin(\phi)}) circle (.12);
  }}
  
  {\color{darkpurple}
  \foreach \phi in {0,45,...,360}{
    \draw[line width=0.1mm,fill=evenlesspalepurple] ({.8+1.025*cos(\phi)},{-6.2+1.025*sin(\phi)}) circle (.08);
  }}

{\color{darkblue}
\foreach \psi in {0, 45,..., 360}{
  \foreach \phi in {0,22.5,...,360}{
    \draw[line width=0.02mm,fill=evenlesspaleblue] ({.8+1.025*cos(\psi) + .1*cos(\phi)},{-6.2+1.025*sin(\psi) + .1*sin(\phi)}) circle (.01);
  }}}

  {\color{darkpurple}
  \foreach \phi in {0,2.8125,...,360}{
    \draw[line width=0.01mm,fill=evenlesspalepurple] ({.8+.76*cos(\phi)},{-6.2+.76*sin(\phi)}) circle (.01);
  }}

   {\color{darkred}
  
    \node at (-2.2, -5.05) {\color{darkred}$\scriptstyle  H^{k_0+1}_{M_{k_0}}, \,\,H^{k_0+2}_{M_{k_0}}$};
  
    \draw[line width=0.15mm, ->] (-1.25,-5.35) -- (0.0,-5.87);
  }

    {\color{darkblue}
  
    \node at (3.4, -4.2) {\color{darkblue}$\scriptstyle  G^{k_0+1}_{M_{k_0}}, \,\,G^{k_0+2}_{M_{k_0}}$};
  
    \draw[line width=0.15mm, ->] (2.6,-4.4) -- (1.63,-5.37);
  }


\end{tikzpicture}   
\caption{The setup for Claim \ref{Hmt}. The next figure shows a blow-up of the picture at time $M_{k_0}$}  \label{HmtPicture}
\end{figure}

\begin{claimproof} (of Claim \ref{Hmt}) To prove {\emph 1.}, we first assume $m = M_{k_0}$ for some $k_0$ after which the general case will follow easily. By \emph{3.} of Claim \ref{InvariantAnnuliDiscs} with $k=k_0+1$, on taking preimages under $Q_{M_{k_0}, M_{k_0+1}-1}$, we have 
${\mathrm D}(0, \tfrac{1}{2}) \subset (Q_{M_{k_0}, M_{k_0+1}-1})^{-1}(H^{k_0+1}_{M_{k_0+1}-1})$. Since by Definition \ref{GmkHmkDef} we have   $H^{k_0+1}_{M_{k_0}} = (Q_{M_{k_0}, M_{k_0+1}-1})^{-1}(H^{k_0+1}_{M_{k_0+1}-1})$, we must then have ${\mathrm D}(0, \tfrac{1}{2}) \subset H^{k_0+1}_{M_{k_0}}$. 

In exactly the same way, we see that at time $M_{k_0+1}$ we also have ${\mathrm D}(0, \tfrac{1}{2}) \subset H^{k_0+2}_{M_{k_0+1}}$. If we now take preimages under the composition $Q_{M_{k_0}, M_{k_0+1}}$ and use \emph{2.} of Claim  \ref{InvariantAnnuliDiscs}, we have that 
$${\mathrm D}(0, \tfrac{1}{2}) \subset (Q_{M_{k_0}, M_{k_0+1}})^{-1}({\mathrm D}(0, \tfrac{1}{2})) \subset (Q_{M_{k_0}, M_{k_0+1}})^{-1}(H^{k_0+2}_{M_{k_0+1}}).$$

The connected set ${\mathrm D}(0, \tfrac{1}{2})$ at time $M_{k_0}$ thus lies in a single connected component of $(Q_{M_{k_0}, M_{k_0+1}})^{-1}(H^{k_0+2}_{M_{k_0+1}})$.  However, we already saw above that 
${\mathrm D}(0, \tfrac{1}{2}) \subset H^{k_0+1}_{M_{k_0}}$ and by \emph{1.} of Lemma \ref{GmkHmkDecr} we have that $H_{M_{k_0+1}-1}^{k_0+2} \subset H_{M_{k_0+1}-1}^{k_0+1}$. If we then take the preimage under $Q_{M_{k_0},M_{k_0+1}-1}$ and recall that by Definition \ref{GmkHmkDef} we have $H_{M_{k_0}}^{k_0+1} = (Q_{M_{k_0},M_{k_0+1}-1})^{-1}(H_{M_{k_0+1}-1}^{k_0+1})$, we see that every component of $(Q_{M_{k_0},M_{k_0+1}-1})^{-1}(H_{M_{k_0+1}-1}^{k_0+2})$ must be contained in $H_{M_{k_0}}^{k_0+1} $. We can then conclude that ${\mathrm D}(0, \tfrac{1}{2})$ lies in a single component of 
$(Q_{M_{k_0}, M_{k_0+1}})^{-1}(H^{k_0+2}_{M_{k_0+1}})$ which in turn lies in $H^{k_0+1}_{M_{k_0}}$. We now turn to showing that this single component of $(Q_{M_{k_0}, M_{k_0+1}})^{-1}(H^{k_0+2}_{M_{k_0+1}})$ must in fact be $H^{k_0+2}_{M_{k_0}}$.

Recall that in view of Claim \ref{HmkConn}, Definition \ref{GmkHmkDef} (applied to $H^{k_0+2}_{M_{k_0+1}}$), and the backward invariance condition \eqref{Invariance3}, $H^{k_0+2}_{M_{k_0+1}}$ is a connected subset of $\overline {\mathrm D}(0,2)$.
By Definition \ref{GmkHmkDef} (applied this time to $G^{k_0+1}_{M_{k_0+1}-1}$ and $H^{k_0+1}_{M_{k_0+1}-1}$), the preimage of this set under $P_{M_{k_0+1}}$ consists of two components - one in the set $G^{k_0+1}_{M_{k_0+1}-1}$ and the other in the set $H^{k_0+1}_{M_{k_0+1}-1}$. Since ${\mathrm D}(0, \tfrac{1}{2}) \subset H^{k_0+2}_{M_{k_0+1}}$ from above while clearly $0 = P_{M_{k_0+1}}(a_{k_0+1}) \in {\mathrm D}(0, \tfrac{1}{2})$, this second component (i.e. the one contained in $H^{k_0+1}_{M_{k_0+1}-1}$) will contain the single critical value $a_{k_0+1} = - \sqrt{-b_{k_0+1}}$ of $Q_{M_{k_0}, M_{k_{0}+1}-1}$. On taking preimages under this polynomial of degree $2^{m_{k_0}}$ and using the fact that, from Definition \ref{GmkHmkDef} we have $G^{k_0+1}_{M_{k_0}} = (Q_{M_{k_0}, M_{k_{0}+1}-1})^{-1}(G^{k_0+1}_{M_{k_0+1}-1})$ and $H^{k_0+1}_{M_{k_0}}= (Q_{M_{k_0}, M_{k_{0}+1}-1})^{-1}(H^{k_0+1}_{M_{k_0+1}-1})$, we see that $(Q_{M_{k_0}, M_{k_0+1}})^{-1}(H^{k_0+2}_{M_{k_0+1}})$ consists of $2^{m_{k_0}}$ components which lie in $G^{k_0+1}_{M_{k_0}}$ and a single additional component lying inside $H^{k_0+1}_{M_{k_0}}$. 

Since by Definition \ref{GmkHmkDef} $H^{k_0+2}_{M_{k_0}}$ is the preimage of $H^{k_0+2}_{M_{k_0+1}-1} $ (under $Q_{M_{k_0}, M_{k_0+1}-1}$) while by Lemma \ref{GmkHmkDecr} it is a subset of $H^{k_0+1}_{M_{k_0}}$, this last component of $(Q_{M_{k_0}, M_{k_0+1}})^{-1}(H^{k_0+2}_{M_{k_0+1}})$ at time $M_{k_0}$ is the unique component which lies inside $H^{k_0+1} _{M_{k_0}}$ and so must 
then be the (connected) set $H^{k_0+2}_{M_{k_0}}$. Thus it must be the case that ${\mathrm D}(0, \tfrac{1}{2}) \subset H^{k_0+2}_{M_{k_0}}$.

In exactly the same way, one can deduce that at time $M_{k_0+1}$ we have  ${\mathrm D}(0, \tfrac{1}{2}) \subset H^{k_0+3}_{M_{k_0+1}}$ and a very similar argument to the above allows us to then deduce that  ${\mathrm D}(0, \tfrac{1}{2}) \subset H^{k_0+3}_{M_{k_0}}$. Continuing inductively in this way, we have that ${\mathrm D}(0, \tfrac{1}{2}) \subset H^{k}_{M_{k_0}}$ for every $k > k_0$ and, by Lemma \ref{HmtInvariance1}, it follows that ${\mathrm D}(0, \tfrac{1}{2}) \subset \tilde {\mathcal H}_{M_{k_0}}$ and so ${\mathrm D}(0, \tfrac{1}{2}) \subset \interior \,\tilde {\mathcal H}_{M_{k_0}}$. Thus $\interior \,\tilde {\mathcal H}_{M_{k_0}}$ is non-empty and it then follows immediately from \emph{2.} of Lemma \ref{HmtInvariance1} and the open mapping theorem that $\interior \,\Hmt$ is non-empty for any $m \ge 0$ and that in addition it also readily follows that $a_k \in \tilde {\mathcal H}_m$ if $m = M_k -1$ for some $k$ while $0 \in \tilde {\mathcal H}_m$ for all other values of $m$.

To show that $\interior \,\Hmt \subset \Fm$, by definition of $\Hmt$ as well as the sets $H_0^k$, ${\mathcal S}_k$ (Definitions \ref{SkDef}, \ref{GmkHmkDef}, \ref{HmtDef}), we have $\Hmt \subset Q_m(\cap_{k \ge 1} {\mathcal S}_k)$. Hence, by \eqref{Interior} in the statement of Theorem \ref{ThmJm},  $\interior\, \Hmt \subseteq \interior \,Q_m(\cap_{k \ge 1} {\mathcal S}_k) \subseteq \interior \cap_{k \ge 1} Q_m({\mathcal S}_k) \subseteq \Fm$ while we showed above that $\interior \,\Hmt \neq \emptyset$, which proves \emph{1.} 


\begin{figure}
\vspace{.5cm}
\begin{tikzpicture}

    {\color{darkred}
\draw[line width=0.1mm, fill=palepink] (0,0) circle (6);}

    {\color{darkred}
\draw[line width=0.1mm, fill=lesspalepink] (0,0) circle (4);}

    {\color{darkred}
\draw[line width=0.1mm, fill=evenlesspalepink] (0,0) circle (3.77);}

{\color{darkblue}
  \foreach \phi in {0,45,...,360}{
    \draw[line width=0.05mm,fill=lesspaleblue] ({5*cos(\phi)},{5*sin(\phi)}) circle (.5);
  }}
  
  {\color{darkpurple}
  \foreach \phi in {0,45,...,360}{
    \draw[line width=0.05mm,fill=evenlesspalepurple] ({5*cos(\phi)},{5*sin(\phi)}) circle (.33);
  }}
  
  {\color{darkblue}
\foreach \psi in {0, 45,..., 360}{
  \foreach \phi in {0,22.5,...,360}{
    \draw[line width=0.05mm,fill=evenlesspaleblue] ({5*cos(\psi) + .4167*cos(\phi)},{5*sin(\psi) + .4167*sin(\phi)}) circle (.042);
  }}}
  
  {\color{darkpurple}
  \foreach \phi in {0,2.8125,...,360}{
    \draw[line width=0.05mm,fill=evenlesspalepurple] ({3.885*cos(\phi)},{3.885*sin(\phi)}) circle (.042);
  }}

    {\color{black}
  \draw[line width=0.2mm] (-.08,-.08) -- (.08,.08);
  \draw[line width=0.2mm] (-.08,.08) -- (.08,-.08);
  
  \node at (.22,-.22) {\footnotesize $0$};
 
  }

 {\color{darkred}  
 
   \node at (7.7, -0.05) {\color{darkred} \footnotesize$ H^{k_0}_{M_{k_0}} = \overline {\mathrm{D}}(0,2)$};
   
   \draw[line width=0.25mm, ->] (6.45,0) -- (6.05,0.0);
   
     \node at (6.8, 1.7) {\color{darkred} \footnotesize$ H^{k_0+1}_{M_{k_0}}$};
     
      \draw[line width=0.25mm, ->] (6.2,1.7) -- (3.9,1.1);

       \node at (6.2, 3.2) {\color{darkred} \footnotesize$ H^{k_0+2}_{M_{k_0}}$};
      
      \draw[line width=0.25mm, ->] (5.7,3.1) -- (3.36,1.8); 
 }

 {\color{darkblue}
 
  \node at (5.1, 4.7) {\color{darkblue} \footnotesize$ G^{k_0+1}_{M_{k_0}}$};

  \draw[line width=0.25mm, ->] (4.6,4.6) -- (3.93,3.93);
  
  \node at (3.7, 5.75) {\color{darkblue} \footnotesize$ G^{k_0+2}_{M_{k_0}}$};
  
  \draw[line width=0.25mm, ->] (3.55,5.5) -- (3.55,3.95);
 }

\end{tikzpicture}   
\caption{A blow-up of the picture at time $M_{k_0}$}  \label{Mk0Picture}
\end{figure}
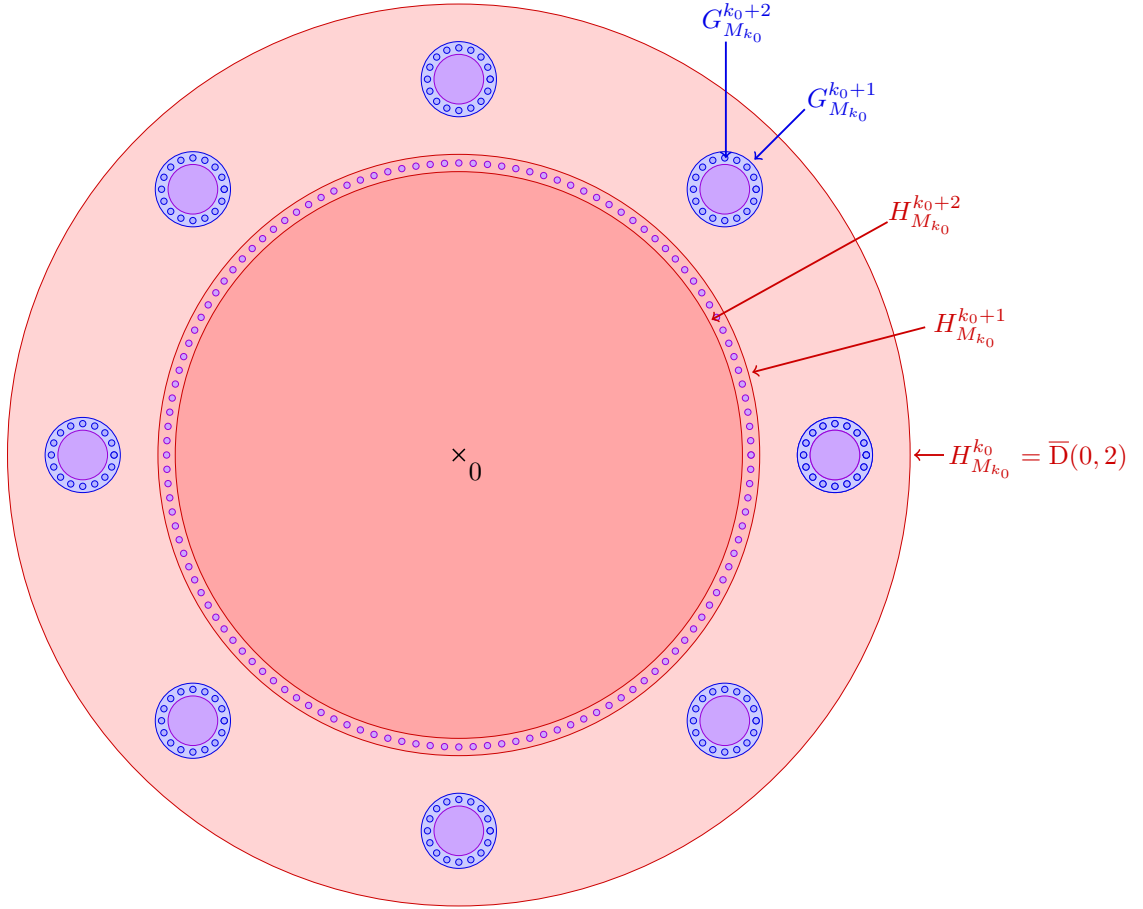


To prove \emph{2.}, let $m \ge 0$, let $z \in \partial \Hmt \subset \Hmt$ and pick $k_0$ as small as possible so that $m \le M_{k_0} - 1$. Suppose for the sake of a contradiction that $z \in \Fm$, the Fatou set at time $m$. Then $z \in U$ for some Fatou component $U$ at this time $m$. Also, since $z \in \Hmt$, by \emph{2.} of Lemma \ref{HmtInvariance1} and Definitions \ref{GmkHmkDef}, \ref{HmtDef}, at the times $M_k$, $k \ge k_0$ we have $Q_{m, M_k}(z) \in  \tilde {\mathcal H}_{M_k}  \subset H^k_{M_k} = \overline {\mathrm D}(0, 2)$. Thus the orbit of $z$ remains bounded at these times. Since $U$ is a Fatou component and thus path connected, it follows by a routine argument that the orbits of all points in $U$ must remain bounded, at least at the times $M_k$. Finally, by complete invariance (Theorem \ref{ThmCompInvar}) we clearly have that at time $M_k$ the Fatou component $U_{M_k} = Q_{m, M_k}(U)$ must meet $\overline {\mathrm D}(0, 2)$. 

At the immediately preceding times $M_k -1$ (provided $m < M_k$) the iterated Fatou component $U_{M_k-1} = Q_{m, M_k-1}(U)$ must then meet (at least one of the) two disjoint preimages $G^k_{M_k -1}$, $H^k_{M_k -1}$ of this disc. However, any point outside these two sets will be mapped outside $\overline {\mathrm D}(0, 2)$ 
by $P_{M_k}$ and its orbit will then escape to infinity in view of \eqref{Invariance1}, \eqref{Invariance2}. Since we saw that all points of $U$ have bounded orbits (at least at the times $M_k$), it follows that the connected set $U_{M_k-1}$ lies entirely inside $G^k_{M_k -1}$ or $H^k_{M_k -1}$ and since $z \in \Hmt$, again by \emph{2.} of Lemma \ref{HmtInvariance1}, it follows that  $Q_{m, M_k-1}(z) \in \tilde {\mathcal H}_{M_k -1} \subset H^k_{M_k-1}$ so that $U_{M_k-1}$ must be a subset of $H^k_{M_k -1}$. 

It follows from this using Theorem \ref{ThmCompInvar}, on taking preimages under $Q_{m, M_{k_0 -1}}$ and recalling the definition of $H^{k_0}_m$ (Definition \ref{GmkHmkDef}) together with the minimality of $k_0$, that 

$$U \subset (Q_{m,M_{k_0}-1})^{-1}(U_{M_{k_0}-1}) \subset (Q_{m,M_{k_0}-1})^{-1}(H_{M_{k_0}-1}^{k_0}) = H_m^{k_0}.$$

At time $M_{k_0}$, if we replace $k_0$ by $k_0 +1$ and $U$ at time $m$ by $U_{M_{k_0+1}}$ we likewise have that 

\begin{eqnarray*}
\hspace{2cm}U_{M_{k_0}} &\subset& (Q_{M_{k_0},M_{k_0+1}-1})^{-1}(U_{M_{k_0+1}-1})\\ 
&\hspace{1.2cm}& \subset (Q_{M_{k_0},M_{k_0+1}-1})^{-1}(H_{M_{k_0+1}-1}^{k_0+1}) = H_{M_{k_0}}^{k_0+1}.\\ 
\end{eqnarray*}

Recalling that the set $H^{k_0 + 1}_{M_{k_0}}$ is path connected and thus connected in view of Claim \ref{HmkConn}, 
taking preimages under $P_{M_{k_0}}$, we see that the connected set $U_{M_{k_0}-1}$ belongs to one of the two connected components of $(P_{M_{k_0}})^{-1}(H^{k_0 + 1}_{M_{k_0}})$. By \emph{1.} of Lemma \ref{GmkHmkDecr} and Definition \ref{GmkHmkDef}, $H^{k_0 + 1}_{M_{k_0}} \subset  H^{k_0}_{M_{k_0}} = \overline {\mathrm D}(0,2)$ so that one of these preimage components lies in $G^{k_0}_{M_{k_0}-1}$ while the other lies in $H^{k_0}_{M_{k_0}-1}$. However, as above since $z \in \Hmt$, by \emph{2.} of Lemma \ref{HmtInvariance1},  $Q_{m, {M_{k_0}-1}}(z) \in  \tilde {\mathcal H}_{M_{k_0}-1}$, so that by Definition \ref{HmtDef} $U_{M_{k_0}-1}$ lies in that preimage component of $H^{k_0 + 1}_{M_{k_0}}$ which is in $H^{k_0 }_{M_{k_0}-1}$ which (from Definition \ref{GmkHmkDef} of the sets $H^k_m$) must then be $H^{k_0+1}_{M_{k_0}-1}$ and on again taking preimages under $Q_{m, M_{k_0-1}}$ and applying Definition \ref{GmkHmkDef} we then have that $U \subset H^{k_0+1}_m$. Continuing inductively in a similar way we see that $U \subset H^k_m$ for all $k \ge k_0$ and by \emph{1} of Lemma \ref{HmtInvariance1}, it follows that $U \subset \Hmt$. Since $U$ is open, $z \in \interior \, \Hmt$ which contradicts our assumption that $z \in \partial \Hmt$, which proves \emph{2.}


For \emph{3.} recall that by \emph{2.} of Lemma \ref{HmtInvariance1}, $P_{m+1}(\Hmt) = \tilde {\mathcal H}_{m+1}$. By the open mapping theorem $P_{m+1}(\interior \,\Hmt) \subset \interior\, \tilde {\mathcal H}_{m+1}$. On the other hand, by \emph{1.}, \emph{2.} above, \emph{2.} of Lemma \ref{HmtInvariance1}, and Theorem \ref{ThmCompInvar}, $P_{m+1}(\partial \Hmt) \subset \partial \tilde {\mathcal H}_{m+1}$. From this and \emph{2.} of Lemma \ref{HmtInvariance1} again, it follows immediately that we have equality in both of these and $P_{m+1}(\interior \,\Hmt) = \interior\, \tilde {\mathcal H}_{m+1}$ and $P_{m+1}(\partial \Hmt) =\partial \tilde {\mathcal H}_{m+1}$ as desired. 

For \emph{4.} recall that by Lemma \ref{GmkHmkDecr} and Claim \ref{HmkConn} $\Hmt$ is connected as it is an intersection of nested compact connected sets. Suppose now that $\chat \setminus \Hmt$ has a bounded complementary component which we call $U$ and again let $k_0$ be as small as possible such that $m \le M_{k_0} - 1$. Then $\partial U \subset \Hmt$ and at the times $M_k$, $k \ge k_0$, since by \emph{2.} of Lemma \ref{HmtInvariance1} $Q_{m,M_k}(\Hmt) = \tilde {\mathcal H}_{M_k} \subset H^k_{M_k} = \overline {\mathrm D}(0,2)$, it follows by Theorem \ref{ThmCompInvar} and the maximum principle that $U_{M_k} = Q_{m, M_k}(U) \subset \overline {\mathrm D}(0,2)$. By Lemma \ref{JSubseq} we must in fact have that 
$U \subset \Fm$, the Fatou set at time $m$. A similar argument as in the in the proof of \emph{2.} (which requires that $U$ be a subset of $\Fm$) then shows that $U \subset \interior \, \Hmt \subset \Hmt$ which then contradicts our assumption that $U$ was a bounded component of $\chat \setminus \Hmt$ (we remark that the main difference is that here we no longer have a point $z \in U$ which lies in $\partial \Hmt \subset \Hmt$, but we do have that $\partial U$ is adherent to $\Hmt$ and using \emph{2.} of Lemma \ref{HmtInvariance1} this is sufficient to ensure that at any time $n \ge m$, $\overline {Q_{m,n}(U)} \cap \Hnt \ne \emptyset$ so that the argument still goes through).

Thus, $\chat \setminus \Hmt$ consists of a single unbounded component and is thus open and connected. Since from above $\Hmt$ is connected, by Theorem 2.2 on page 202 of \cite{Con}, $\chat \setminus \Hmt$ is simply connected whence by 4.3 on page 144 of \cite{New}, $\partial \Hmt$ is connected and therefore trivially uniformly perfect and so we have proved Claim \ref{Hmt}.
\end{claimproof}

We are now in a position to prove \emph{3.} of the statement of Theorem \ref{MainTh1} and also, eventually, the remaining parts. For $m \ge 0$, if $k_0$ is as small as possible so that $m \le M_{k_0}$, we define

\begin{equation}
\label{HmDef1}
\Hm  = \bigcup_{k=k_0}^\infty{Q_{m,M_k}^{-1}(\partial \tilde{\mathcal H}_{M_k})} = \bigcup_{n \ge m} Q_{m,n}^{-1}(\partial \Hnt),
\end{equation}

the second equality here following from \emph{3.} of Claim \ref{Hmt} since if $n < M_k$, then $\partial \Hnt \subset Q_{n, M_k}^{-1}(Q_{n, M_k}(\partial \Hnt)) = Q_{n, M_k}^{-1}(\partial \tilde {\mathcal H}_{M_k})$. We draw the reader's attention to the fact that we are taking preimages of the boundaries $\partial \Hnt$ rather than the sets $\Hnt$ themselves, the reason for this clearly being that we want $\Hm$ to be a subset of the iterated Julia set $\Jm$. Also, by \emph{1.} of Lemma \ref{HmtInvariance1}, we then immediately have that $\Hm \neq \emptyset$ for each $m \ge 0$.  

For $m \ge 0$ and $n > m$, it follows from \emph{4.} of Claim \ref{Hmt} above and the fact that $Q_{m,n}$ is continuous that $Q_{m,n}^{-1}{(\partial \Hnt)}$ is compact and connected and thus closed and uniformly perfect. Moreover, this set is also a   
subset of the Julia set $\Jm$ in view of Theorem \ref{ThmCompInvar} and \emph{2.} of the same claim. Lastly, one checks easily that 
the sets $Q_{m,n}^{-1}(\partial \Hnt)$ are increasing in $n$ by \emph{3.} of Claim \ref{Hmt} above so that $\Hm$ is an $\Fsigma$ subset of $\Jm$ which from above is an increasing union of closed uniformly perfect sets. 

To show \emph{3.}, it only remains to show $\Hm$ is dense in $\Jm$. First, we prove the following claim. 

\begin{claim}
\label{Hm1}
$\Hm \subset \Jm$ and is precisely the set of points $z \in \Jm$ for which $Q_{m, M_k-1}(z) \in H^k_{M_k -1}$ for all but finitely many $k$.
\end{claim}

\begin{claimproof} (of Claim \ref{Hm1})
That $\Hm \subset \Jm$ follows by part \emph{2.} of Claim \ref{Hmt} above, combined with complete invariance (Theorem \ref{ThmCompInvar}). 

Suppose now that $z \in \Hm$ so that $Q_{m,n}(z) \in \partial \Hnt \subset \Hnt$ for some $n \ge m$. By part \emph{2.} of Lemma \ref{HmtInvariance1} together with the definition of the sets $\Hmt$ (Definition \ref{HmtDef}) given before that lemma, if $k_0$ is the smallest integer such that $M_{k_0} -1 \ge n$, then 
$Q_{m,M_k - 1}(z) \in \tilde {\mathcal H}_{M_k-1} \subset H^k_{M_k-1}$ for every $k \ge k_0$. 

Conversely, if $z \in \Jm$ is such that 
$Q_{m,M_k - 1}(z) \in H^k_{M_k-1}$ for all but finitely many $k$, then there is an integer $k_0$ with $M_{k_0} - 1 \ge m$ such that  $Q_{m,M_k - 1}(z) \in H^k_{M_k-1}$ for all $k \ge k_0$. Hence, at time $M_{k_0}$, $Q_{m, M_{k_0}}(z)$ belongs to a component of $Q_{M_{k_0},M_k-1}^{-1}(H^k_{M_k -1})$ for each $k > k_0$. However, we saw earlier in the proof of part \emph{1.} of Claim \ref{Hmt} that there is only one component of this preimage which lies inside $H^{k_0}_{M_{k_0}}$, namely $H^k_{M_{k_0}}$.  We then have by \emph{1.} of Lemma  \ref{HmtInvariance1}
that $Q_{m, M_{k_0}}(z) \in \tilde {\mathcal H}_{M_{k_0}}$ so that $z \in Q_{m,M_{k_0}}^{-1}( \tilde {\mathcal H}_{M_{k_0}})$. However, since $z \in \Jm$, by Theorem \ref{ThmCompInvar}, \emph{1.} of  Claim \ref{Hmt}, and \eqref{HmDef1}, we must also have that $z \in Q_{m,M_{k_0}}^{-1}( \partial \tilde {\mathcal H}_{M_{k_0}}) \subset \Hm$ which finishes the proof of Claim \ref{Hm1}. 
\end{claimproof}

From the above claim, it is easy to see that, if $z \in \Hm$, then so is any point in $Q_{m,n}^{-1}(Q_{m,n}(z))$ for any fixed  $n \ge m$. Also, it follows from condition \eqref{Invariance2} that the sequence  $\{Q_{M_{k-1}, M_k}\}_{k=1}^\infty$ is of the type in B\"uger's paper \cite{Bug} where we can take $U = \chat \setminus \overline {\mathrm D}(0,2)$ as our invariant neighbourhood of infinity. We can then apply \emph{1.} of B\"uger's result (Theorem \ref{SelfSimilarity}) combined with Lemma \ref{JSubseq} to conclude that $\Hm$ is dense in $\Jm$ so that \emph{3.} in the statement of  Theorem \ref{MainTh1} follows. 

We now proceed to prove \emph{5.} of the statement of Theorem \ref{MainTh1}. By \emph{1.} and \emph{2.} of Claim \ref{Hmt}, $\Hmt$ contains at least one bounded Fatou component in its interior. If $n > m$, then by \emph{1.}, \emph{2.}, and \emph{3.} of Claim \ref{Hmt}
\begin{equation}
\label{HmtInvariance2}
Q_{m,n} (\partial \Hmt) = \partial \Hnt \subset \Jn,  \qquad Q_{m,n}(\interior \,\Hmt) = \interior \,\Hnt \subset \Fn.
\end{equation}

By definition of the sequence $\Pm$, the critical value of $P_m$ is either the critical value $a_k$ of $P_{M_k -1}$ if $m = M_k-1$ for some $k \ge 1$, the critical value $b_k$ of $P_{M_k}$  if $m = M_k$ for some $k \ge 1$, or the critical value $0$ of $P_m= z^2$ at all other remaining times $m \ge 1$. Dealing first with the critical values $b_k$, recall that, since we had assumed in the discussion before the definitions of the sets $G^k_m$, $H^k_m$ (Definition \ref{GmkHmkDef}) that $b_k < -6$, while at the times $M_k$, $k \ge 1$, we have (using Definitions \ref{GmkHmkDef}, \ref{HmtDef}) that $ \tilde {\mathcal H}_{M_k} \subset H^k_{M_k} = \overline {\mathrm D}(0,2)$, we must have $b_k \notin \tilde {\mathcal H}_{M_k}$. On the other hand, in view of \emph{1.} of Claim \ref{Hmt}, $a_k \in \tilde {\mathcal H}_m$ if $m = M_k -1$ for some $k$ while $0 \in \tilde {\mathcal H}_m$ for all other remaining values of $m$ which are not equal to either $M_k -1$ or $M_k$ for some value of $k$.

Combining all of this, we see that at any given time $m \ge 1$, we have one of the following three possibilities:

\begin{enumerate}
\item If $m \ne M_k$ and $m \ne M_k -1$ for any $k \ge 1$, the critical value $0$ of $P_m$ is in $\Hmt$ and is the image under $P_m$ of the (single) critical point $0$ of $P_m$ and this critical point lies in $\tilde {\mathcal H}_{m-1}$,

\item If $m = M_k -1$ for some $k \ge 1$, the critical value $a_k$ of $P_m$ is in $\Hmt$ and is the image under $P_m$ of the (single) critical point $0$ of $P_m$ and this critical point lies in $\tilde {\mathcal H}_{m-1}$,

\item If $m = M_k$ for some $k \ge 1$, the critical value $b_k$ of $P_m$ satisfies $|b_k| > 6$ and does not meet $\Hmt \subset \overline {\mathrm D}(0,2)$ and consequently the critical point $0$ of $P_m$ does not meet $\tilde {\mathcal H}_{m-1}$.
\end{enumerate}

Recall that by \emph{4.} of Claim \ref{Hmt} $\tilde {\mathcal H}_{M_k}$ is connected while by by \emph{2.} of Lemma \ref{HmtInvariance1} $Q_{m,M_k}^{-1}(Q_{m,M_k}(\Hmt)) = Q_{m,M_k}^{-1}(\tilde {\mathcal H}_{M_k})$. Note also that from above $\Hmt$ contains the critical value of $P_m$ except at the times $M_k$ when there are two univalent branches of $P_{M_k}^{-1}$ defined on neighbourhood of the closed disc $\overline {\mathrm D}(0,2)$ , one mapping $\overline {\mathrm D}(0,2)$ to $G^k_{M_k-1}$ and one mapping the same disc to $H^k_{M_k-1}$ (using Definition \ref{GmkHmkDef}).

Thus, by repeatedly taking preimages one (quadratic) polynomial at a time,  
we see that, for each $k_1$ with $M_{k_1} > m$, one of the components of $Q_{m,M_{k_1}}^{-1}(Q_{m,M_{k_1}}(\Hmt))$ is $\Hmt$. For each of the other components of $Q_{m,M_{k_1}}^{-1}(Q_{m,M_{k_1}}(\Hmt))$, there is some $k_0 \le k_1$ with $m < M_{k_0}$ for which this component is the image of $\tilde {\mathcal H}_{M_{k_0}}$ under a univalent branch of $Q_{m,M_{k_0}}^{-1}$ defined on a neighbourhood of the set $\tilde{\mathcal H}_{M_{k_0}}$.

Note that, by \emph{4.} of Claim \ref{Hmt}, $\chat \setminus \tilde{\mathcal H}_{M_{k_0}}$ is simply connected which allows us by a routine argument using a Riemann map to ensure that this neighbourhood of $\tilde{\mathcal H}_{M_{k_0}}$ is simply connected so that the univalent branches of $Q_{m,M_{k_0}}^{-1}$ are well defined in view of the monodromy theorem. We summarize all of this below.

\vspace{.2cm}
\begin{claim}
\label{InverseImageComponents}
For each $m \ge 0$ and each $k_1\ge 1$ with $M_{k_1} > m$, each component of $Q_{m,M_{k_1}}^{-1}(Q_{m,M_{k_1}}(\Hmt)) = Q_{m,M_{k_1}}^{-1}(\tilde{\mathcal H}_{M_{k_1}})$ is either $\Hmt$ or is the image of 
$\tilde{\mathcal H}_{M_{k_0}}$ under a well-defined univalent branch of $Q_{m,M_{k_0}}^{-1}$ defined on a simply connected neighbourhood of the set $\tilde{\mathcal H}_{M_{k_0}}$ for some $k_0 \le k_1$ with $m < M_{k_0}$.
\end{claim}

Recall that by \emph{1.} of Claim \ref{Hmt} $\interior \,\Hmt$ is non-empty and contained in the Fatou set $\Fm$. 
If we now let $U$ be one of these components of $\interior \,\Hmt$, then $U$ is contained in $\Hmt$ and thus enclosed (using \emph{2.} of Claim \ref{Hmt}) by $\partial \Hmt \subset \Jm$. $U$ is therefore a bounded component of the Fatou set $\Fm$ and by Theorem \ref{ThmCompInvar}, if $k \ge 1$ with $m < M_k$,  then $Q_{m,M_k}(U)$ is then a bounded Fatou component (of the Fatou set ${\mathcal F}_{M_k}$ at time $M_k$). By \emph{1.} and \emph{2.} of Claim \ref{Hmt} and the maximality of Fatou components, $Q_{m,M_k}(U)$ is also a component of $\interior \, \tilde {\mathcal H}_{M_k}$ which using \eqref{HmtInvariance2} is then enclosed by $Q_{m,M_k}(\partial \Hmt) = \partial \tilde {\mathcal H}_{M_k} \subset {\mathcal J}_{M_k}$. 

 Recall Claim \ref{InverseImageComponents} above which states that each $k_1\ge 1$ with $M_{k_1} > m$ each component of $Q_{m,M_{k_1}}^{-1}(Q_{m,M_{k_1}}(\Hmt)) = Q_{m,M_{k_1}}^{-1}(\tilde{\mathcal H}_{M_{k_1}})$ is either $\Hmt$ or is the image of 
$\tilde{\mathcal H}_{M_{k_0}}$ under a well-defined univalent branch of $Q_{m,M_{k_0}}^{-1}$ defined on a simply connected neighbourhood of the set $\tilde{\mathcal H}_{M_{k_0}}$ for some $k_0 \le k_1$ with $m < M_{k_0}$. 
Recall that by \emph{4.} of Claim \ref{Hmt} that $\partial \tilde {\mathcal H}_{M_{k_1}}$ is connected. It then follows from 
\emph{3.} of this same claim that each component of $Q_{m,M_{k_1}}^{-1}(Q_{m,M_{k_1}}(\partial \Hmt)) = Q_{m,M_{k_1}}^{-1}(\partial \tilde {\mathcal H}_{M_{k_1}})$ is either $\partial \Hmt$ or the image of $\partial \tilde{\mathcal H}_{M_{k_0}}$ under a well-defined univalent branch of $Q_{m,M_{k_0}}^{-1}$ defined on a simply connected neighbourhood of the set $\tilde{\mathcal H}_{M_{k_0}}$ for some $k_0 \le k_1$ with $m < M_{k_0}$.

By Theorem \ref{ThmCompInvar} and Claim \ref{InverseImageComponents} it then follows from the two paragraphs above that each of the components of $Q_{m,M_{k_1}}^{-1}(Q_{m,M_{k_1}}(\partial \Hmt))$ then encloses a bounded Fatou component which is either $U$ or the image of  the iterated component $Q_{m, M_{k_0}}(U)$ under a well-defined univalent branch of $Q_{m,M_{k_0}}^{-1}$ defined on a simply connected neighbourhood of the set $\tilde{\mathcal H}_{M_{k_0}}$ for some $k_0 \le k_1$ with $m < M_{k_0}$. In addition, as $\partial \Hmt$  and each $\partial \tilde {\mathcal H}_{M_{k_0}}$ are connected in view of \emph{4.} of Claim \ref{Hmt},  each component of $Q_{m,M_{k_1}}^{-1}(\partial \tilde {\mathcal H}_{M_{k_1}})$ is trivially uniformly perfect. 

To finish showing \emph{5.} of Theorem \ref{MainTh1}, we need to show that the number of components of $Q_{m,M_k}^{-1}(\partial \tilde {\mathcal H}_{M_k})$ tends to infinity as $k \to \infty$ and that these components belong to different components of the iterated Julia set $\Jm$ (the issue one needs to avoid here is that the distinct components of $Q_{m,M_k}^{-1}(Q_{m,M_k}(\partial \Hmt))$ could `join up' in the limit as $k$ tends to infinity). To see this, note that, by \emph{3.} of Claim \ref{Hmt} if $m \le M_k$, then $Q_{M_k, M_{k+1}}^{-1}(Q_{m,M_{k+1}}(\partial \Hmt)) = Q_{M_k, M_{k+1}}^{-1}(\partial \tilde {\mathcal H}_{M_{k+1}})$. By a similar discussion to the one used in the proof of \emph{1.} of Claim \ref{Hmt}, using \emph{5.} of this same claim, this set then consists of $2^{m_{k+1}} + 1$ components. Since $|b_k| \to \infty$ as $k \to \infty$, it follows from \eqref{mkbound1} or \eqref{mkbound2} that $m_k \to \infty$ as $k \to \infty$ and so therefore does the number of these components.

To finish the proof of \emph{5.} of Theorem \ref{MainTh1} we make the following definition for a collection of conformal annuli which we will use to ensure the components of $Q_{m,M_k}^{-1}(\partial \tilde {\mathcal H}_{M_k})$ remain distinct as $k$ tends to infinity as well ensuring pointwise thinness of the sets $\Gm$ (\emph{2.} of Theorem \ref{MainTh1}). 

\vspace{.2cm}
\begin{definition}
\label{AkDef}
For each $k \ge 1$ we define the round annulus $A_k$ at time $M_k -1$ as follows:
\vspace{.1cm}
$$A_k = {\mathrm A}(\sqrt{-b_k}, r_k, 2\sqrt{-b_k} - r_k) = {\mathrm A}(-a_k, r_k, 2\sqrt{-b_k} - r_k).$$
 \end{definition}

We remind the reader that $r_k = \sqrt{-b_k} - \sqrt{-b_k -2}$ was introduced in Lemma \ref{rksk} so that we then have $(2\sqrt{-b_k} - r_k) - r_k > 0$ so that $A_k$ is in particular non-empty. Then, again by \emph{1.} of this lemma and Definition \ref{GmkHmkDef}, 

\vspace{-.2cm}
\begin{equation}
\label{InverseImageofDisc}
(P_{M_k})^{-1}(\overline {\mathrm D}(0,2)) = (G_{M_k-1}^k \cup H_{M_k-1}^k) \subset (\overline {\mathrm D}(-\sqrt{-b_k},r_k) \cup \overline {\mathrm D}(\sqrt{-b_k},r_k)). 
\end{equation}

$A_k$ thus contains $G^k_{M_k-1}$ and $H^k_{M_k-1}$ in its bounded and unbounded complementary components respectively, so that it separates these sets. 
Also, since $(P_{M_k})^{-1}(\overline {\mathrm D}(0,2)) = (G_{M_k-1}^k \cup H_{M_k-1}^k)$, in view of \eqref{Invariance1}, \eqref{Invariance2} and Lemma \ref{JSubseq}, we know that this annulus is contained in ${\mathcal A}_{\infty, M_k -1}$, the iterated basin of infinity at time $M_k -1$.

Recall from above that, if $m \ge 0$ is such that $M_{k-1} \le m < M_k -1$ (remember that $M_0=0$ by definition), then $Q_{m, M_k -1}$ maps the critical point $0$ of $P_{m+1}$ to $a_k = - \sqrt{-b_k}$ while $P_{M_k}$ maps $a_k$ to $0$. Also, as remarked near the start of the proof, each $Q_{M_{k-1}, M_k}$ fixes $0$ so that $Q_{M_k-1, m}(a_k)$ is either $a_j$ if $m = M_j -1$ for some $j > k$ or $0$ otherwise. It follows (recalling that we have set $Q_{m,m} = Id$) that 
the critical values of $Q_m$ for $m \ge 1$ are either images of the critical values $a_k$ of $P_{M_k -1}$ for each $k$ with $M_k - 1 \le m$ or of the critical values $b_k$ of $P_{M_k}$ for some $k$ with $M_k \le m$. 

On the other hand, 
since for each $k \ge 1$ $|b_k| > 6 > 4$, by \eqref{Invariance1},  \eqref{Invariance2} (using the points $a_j$, $b_j$ for suitable $j > k$), it follows that, for each $m > M_k$, $|Q_{M_k, m}(b_k)| > 4$. It then follows from this and the above using \eqref{mkbound2} that, for each $k \ge 1$, the critical values of $Q_{M_k -1} = Q_{0, M_k -1}$ (including $a_k$) lie outside the disc ${\mathrm D}(\sqrt{-b_k}, 2\sqrt{-b_k})$ (note that in view of \eqref{mkbound2} we have that $4^{m_k} > 2^{m_k} > 12\sqrt{-b_k} > 3\sqrt{-b_k}$). This in turn implies that, for each $0 \le m \le M_k -1$,  each of the polynomials $Q_{m, M_k -1}$ has a full set of $2^{M_k - m - 1}$ inverse branches defined on this disc. In addition, if $M_{k-1} \le m \le M_k -1$, then these preimages differ by rotations of $\tfrac{2\pi}{2^{M_k - m - 1}}$. We summarize what we have just proved in the following claim.

\begin{claim}
\label{InverseBranches}
For each $k \ge 1$ and $m \le M_k - 1$, the polynomial composition $Q_{m, M_k -1}$ has a full set of $2^{M_k - m - 1}$ univalent inverse branches defined on the disc ${\mathrm D}(\sqrt{-b_k}, 2\sqrt{-b_k})$ and thus on the annulus $A_k = {\mathrm A}(\sqrt{-b_k}, r_k, 2\sqrt{-b_k} - r_k)$.
\end{claim}

Using the claim above, we can make the following definition which will be useful to us in proving \emph{2.} and \emph{5.} of Theorem 1.10. 

\begin{definition}
\label{BkjmDef}
For each $0 \le m \le M_k -1$, and each $1 \le j \le 2^{M_k - m - 1}$, let
$B^{k,j}_m$, be the components of the inverse image of the annulus $A_k$ (defined in Definition \ref{AkDef}) under $Q_{m, M_k -1}$. Each of the sets $B^{k,j}_m$ is thus a conformal annulus of the same modulus as $A_k$.
\end{definition} 

Note that using Lemma \ref{rkskbounds} the modulus of $A_k$ is $\log \tfrac{ 2\sqrt{-b_k} - r_k}{r_k}  > \log (4\sqrt{-b_k} - 1)$ which then tends to infinity as $k \to \infty$ as $b_k \to - \infty$ as $k \to \infty$. In addition, since (as remarked above when we introduced the annuli $A_k$ in Definition \ref{AkDef}) we have $A_k \subset {\mathcal A}_{\infty, M_k -1}$, we must also have that $B^{k,j}_m \subset \Am$ for each $1 \le j \le 2^{M_k - m - 1}$ in view of Theorem \ref{ThmCompInvar}.

In view of the lower bound \eqref{mkbound1} on the integers $m_k$ given in Claim \ref{InvariantAnnuliDiscs} (or the stronger bound \eqref{mkbound2} in the proof of Claim \ref{Contracting1}), for each $k \ge 2$, one checks easily that the set $Q_{M_{k-1},M_k -1}^{-1}({\mathrm D}(\sqrt{-b_k}, 2\sqrt{-b_k}))$ must be contained in  $\overline {\mathrm D}(0, 2)$. If we then take one further preimage under $P_{M_{k-1}}$, it then follows (using Definition \ref{GmkHmkDef}) that 

\vspace{-.2cm}
\begin{equation}
\label{DisjointAnnuli1}
Q_{M_{k-1}-1,M_k -1}^{-1}({\mathrm D}(\sqrt{-b_k}, 2\sqrt{-b_k})) \subset G^{k-1}_{M_{k-1}-1} \cup H^{k-1}_{M_{k-1}-1}.
\end{equation}

We shall see later in Claim \ref{SeparatingAnnuliConfiguration} that this condition implies that for $k_0 \le k_1$ such that $m < M_{k_0} \le M_{k_1}$, the families of conformal annuli $B^{k,j}_m$, $k_0 \le k \le k_1$, $1 \le j \le 2^{M_k - m - 1}$ (as defined above in Definition \ref{BkjmDef}) will in fact be disjoint. It is worth noting here that, although these annuli are disjoint, they can be nested in the sense that one annulus $B^{k,j}_m$ may lie entirely in the bounded complementary component of another (which will correspond to a smaller value of $k$). 

Recall that, by Claim \ref{InverseImageComponents} for each $k_1$ with $M_{k_1} > m$, with the exception of $\Hmt$ itself, all the other components of $Q_{m,M_{k_1}}^{-1}(Q_{m,M_{k_1}}(\Hmt)) = Q_{m,M_{k_1}}^{-1}(\tilde {\mathcal H}_{M_{k_1}})$ 
are the images of $\tilde {\mathcal H}_{M_{k_0}}$ under a univalent branch of $Q_{m, M_{k_0}}^{-1}$ 
defined on a neighbourhood of the set $\tilde {\mathcal H}_{M_{k_0}}$ for some $k_0 \le k_1$ with $m < M_{k_0}$. If we then choose 
$k_1$ with $M_{k_1} > m$ and let $k_0$ be as small as possible with $M_{k_0} >m$,  
one can calculate that there will be precisely 

\[ \sum_{k=k_0}^{k_1}{2^{M_k - m - 1}}\]

 of these other components. 

To finish the proof of \emph{5.} of Theorem \ref{MainTh1}, we require another important claim, the proof of which will reveal much about the structure of the sets $\Hm$ and thus the iterated Julia sets $\Jm$. Recalling the definition of the sets $\Hmt$ as given in Definition \ref{HmtDef}, for each $0 \le m$ and each $k_1$ with $M_{k_1} > m$, again choose $k_0$ to be as small as possible so that $m < M_{k_0} \le  M_{k_1}$. The reader should note that we could possibly have $k_0 = k_1$, for example if $m = M_{k_1}-1$.

\begin{claim}
\label{JoiningTimes}
For $m$, $k_0$, $k_1$ as above and each component $K$ of the inverse image $Q_{m,M_{k_1}}^{-1}(\tilde {\mathcal H}_{M_{k_1}})$, then either $K = \Hmt$ or 
there exists $k_0 \le k \le k_1$ such that the following hold:

\begin{enumerate}
\item $M_k$ is the least integer $j$, $m \le j \le M_{k_1}$ for which we have $Q_{m,j}(K) \cap \tilde {\mathcal H}_j \ne \emptyset$,

\vspace{.2cm}
\item $Q_{m,M_k}(K) = \tilde {\mathcal H}_{M_k}$,

\vspace{.2cm}
\item $Q_{m,M_k-1}(K) \subset G_{M_k -1}^k$.

\end{enumerate}
\end{claim}

\begin{claimproof} (of Claim \ref{JoiningTimes})
First note that, if $K$ is one of the components of $Q_{m,M_{k_1}}^{-1}(\tilde {\mathcal H}_{M_{k_1}})$, then obviously $Q_{m,M_{k_1}}(K) \cap \tilde {\mathcal H}_{M_{k_1}} \ne \emptyset$ and so there clearly exists some integer $j$, $m \le j \le M_{k_1}$ as small as possible for which  $Q_{m,j}(K) \cap \tilde {\mathcal H}_j \ne \emptyset$ as above. 

We now show that we must have either $j = m$ or $j=M_k$ for some $k_0 \le k \le k_1$. To see this, suppose this fails which we then divide into two cases where we either have that $j > M_{k_0}$ or we have $j < M_{k_0}$.

In the first of these possibilities we must then have that $j > M_k$ for some $k_0 \le k < k_1$ where $k$ is as large as possible. Then, again using the backward invariance condition \emph{3.} of Lemma \ref{HmtInvariance1}, one finds that $Q_{m,M_k}(K)$ must contain points of $(Q_{M_k, j})^{-1}(\tilde {\mathcal H}_j) = \tilde {\mathcal H}_{M_k}$. To see this, note first that we consider all possible inverse images when we take the preimage of $\tilde {\mathcal H}_j$ under $Q_{M_k,j}$ to obtain $\tilde {\mathcal H}_{M_k}$. Therefore, since $Q_{m,j}(K) \cap \tilde {\mathcal H}_j \ne \emptyset$, $\tilde {\mathcal H}_{M_k}$ must then contain points of $Q_{m,M_k}(K)$ corresponding to those choices of inverse branches of $Q_{M_k,j}$ by which one obtains points of $Q_{m,M_k}(K)$ from $Q_{m,j}(K)$. Thus $Q_{m,M_k}(K) \cap \tilde {\mathcal H}_{M_k} \neq \emptyset$ and, since $m < M_k < j$, by the minimality of $j$ we then obtain a contradiction. 

In the other case where the condition $j = m$ or $j=M_k$ for some $k_0 \le k \le k_1$ fails, the second possibility is that $m < j < M_{k_0}$. Using a similar argument to above, we then see that $K \cap \Hmt \ne \emptyset$ which again violates the minimality of $j$ and gives a contradiction. 

If $j = m$, then $K \cap \Hmt \ne \emptyset$ and by \emph{2.} of Lemma  \ref{HmtInvariance1} applied to the composition $Q_{m, M_{k_1}}$, we have that $Q_{m, M_{k_1}}(\Hmt) = \tilde {\mathcal H}_{M_{k_1}}$. Then we must have that 
$K$ is all of $\Hmt$ in view of the connectedness of $\Hmt$ using \emph{4.} of Claim \ref{Hmt} and the maximality of $K$ as a connected component of the inverse image 
$Q_{m,M_{k_1}}^{-1}(\tilde {\mathcal H}_{M_{k_1}})$. Otherwise, if $j > m$, we then have that $Q_{m,j}(K) \cap \tilde {\mathcal H}_j \ne \emptyset$. Since, by \emph{2.} of Lemma  \ref{HmtInvariance1} applied to the composition $Q_{j, M_{k_1}}$,
$Q_{m,j}(K)$ is a connected subset of $Q_{j,M_{k_1}}^{-1}(Q_{j,M_{k_1}}(\tilde {\mathcal H}_j )) = Q_{j,M_{k_1}}^{-1}(\tilde {\mathcal H}_{M_{k_1}})$, it follows immediately from Claim \ref{InverseImageComponents} (applied at times $j$ and $M_{k_1}$) that $Q_{m,j}(K) \subset \tilde {\mathcal H}_j$. 

Suppose now for a contradiction that the connected set $Q_{m,j}(K)$ is a proper subset of $\tilde {\mathcal H}_j$.  Using Claim \ref{InverseImageComponents}, (applied at times $m$ and $j$, and recalling that we are now assuming $j \ne m$ so that $j = M_k$ for some $k_0 \le k \le k_1$) we have that $K$ is either a subset of $\Hmt$ or a subset of one of the other components of $Q_{m,j}^{-1}(\tilde {\mathcal H}_j)$ which is the inverse image of $\tilde {\mathcal H}_{M_l}$ for some $k_0 \le l \le k$ under a univalent branch of $Q_{m,M_l}^{-1}$ defined on a simply connected neighbourhood of  $\tilde {\mathcal H}_{M_l}$. 

In the first case, by \emph{2.} of Lemma  \ref{HmtInvariance1} applied to the composition $Q_{m,j}$, $K$ is then a proper subset of the connected set $\Hmt$. However, since again using \emph{2.} of Lemma  \ref{HmtInvariance1}, $Q_{m, M_{k_1}}(\Hmt) = \tilde {\mathcal H}_{M_{k_1}}$, this violates the maximality of $K$ as a component of $Q_{m,M_{k_1}}^{-1}(\tilde {\mathcal H}_{M_{k_1}})$. 

In the second case, let $L$ be the image of the connected set $\tilde {\mathcal H}_{M_l}$ under the univalent branch of $Q_{m,M_l}^{-1}$ defined on a neighbourhood of  $\tilde {\mathcal H}_{M_l}$ as given above. By \emph{2.} of Lemma  \ref{HmtInvariance1} applied for the composition $Q_{M_l,j}$, $Q_{m,M_l}(K)$ must be a proper subset of $\tilde {\mathcal H}_{M_l}$ since otherwise if $Q_{m,M_l}(K) = \tilde {\mathcal H}_{M_l}$, then we would have $Q_{m,j}(K) = Q_{M_l, j}(Q_{m,M_l}(K)) = Q_{M_l, j}(\tilde {\mathcal H}_{M_l}) = \tilde {\mathcal H}_j$ while we had assumed that $Q_{m,j}(K)$ was a proper subset of $\tilde {\mathcal H}_j$. Using the univalence of this branch of $Q_{m,M_l}^{-1}$, we have that 
$K$ is a proper subset of the connected set $L$. However, 

\vspace{-.3cm}
$$Q_{m, M_{k_1}}(L) = Q_{M_l, M_{k_1}}(Q_{m, M_l}(L)) = Q_{M_l, M_{k_1}}(\tilde {\mathcal H}_{M_l}) = \tilde {\mathcal H}_{M_{k_1}}$$

(this last equality following again from \emph{2.} of Lemma  \ref{HmtInvariance1}, this time applied for the composition $Q_{M_l, M_{k_1}}$), this again violates  maximality of $K$ as a component of $Q_{m,M_{k_1}}^{-1}(\tilde {\mathcal H}_{M_{k_1}})$. Hence, in either case we must have $Q_{m,j}(K) = \tilde {\mathcal H}_j$. 

Thus we have that $K = \Hmt$ in the case that $j$ is equal to $m$ (which is the first of the two possibilities in the statement) and also \emph{1.} and \emph{2.} of the second part where we have $j = M_k$ for some $k_0 \le k \le k_1$. 
For \emph{3.} of the second part, recall that $P_{M_k}$ has a critical value $b_k$ with $|b_k| > 6$ while by Definitions \ref{GmkHmkDef} and \ref{HmtDef}, $\tilde {\mathcal H}_{M_k} \subset \overline {\mathrm D}(0,2)$. Recall that $P_{M_k}$ then has two univalent inverse branches defined on a neighbourhood of this disc, and by \emph{3.} of Lemma \ref{HmtInvariance1}, one of these maps $\tilde {\mathcal H}_{M_k}$ to $\tilde {\mathcal H}_{M_k-1}$ while the other (using Definitions \ref{GmkHmkDef} and \ref{HmtDef}) maps $\tilde {\mathcal H}_{M_k}$ inside $G_{M_k-1}^k$. 
It then follows by the minimality of $j=M_k$ that $Q_{m,M_k-1}(K) \subset G^{M_k}$ as desired and with this the proof of Claim \ref{JoiningTimes} is complete. \end{claimproof}

We call the time $M_k$ (or $m$) in the statement of Claim \ref{JoiningTimes} above the \emph{joining time} for a component $K$ of $Q_{m,M_{k_1}}^{-1}(\tilde {\mathcal H}_{M_{k_1}})$. We will use this and the above claim to prove the following important property which will allow us to finish the proof of \emph{5.} of the statement of Theorem \ref{MainTh1}. Recall that if $m \ge 0$, $k \ge 1$ are such that $m \le M_k -1$, we had the annuli $B^{k,j}_m$, $1 \le j \le 2^{M_k - m - 1}$ which we introduced in Definition \ref{BkjmDef} and which were the $2^{M_k - m - 1}$ components of the 
inverse image of the annulus $A_k = {\mathrm A}(\sqrt{-b_k}, r_k, 2\sqrt{-b_k} - r_k)$ (introduced in Definition \ref{AkDef}) under $Q_{m, M_k -1}$ and which each had the same modulus as $A_k$.

\begin{claim}
\label{SeparatingAnnuliConfiguration}
For each $m$, $k_0$, and  $k_1$ as defined above just before the statement of Claim \ref{JoiningTimes}, the conformal annuli $B^{k,j}_m$, $k_0 \le k \le k_1$, $1 \le j \le 2^{M_{k_1} - m - 1}$ are pairwise disjoint and separate the components of $Q_{m,M_{k_1}}^{-1}({\mathcal H}_{M_{k_1}})$ in the sense that any two components of this set will lie in different complementary components of an annulus $B^{k,j}_m$ for some suitable $k$, $j$.
\end{claim}

The reader might find it helpful to consult with Figure \ref{SeparatingAnnuliConfigurationPicture} while going through the proof below. 

\begin{claimproof} (of Claim \ref{SeparatingAnnuliConfiguration}) We proceed by induction on the integer $k$ associated with the joining times $M_k$, $k_0 \le k \le k_1$ in Claim \ref{JoiningTimes} above. Again by this claim, once we have dealt with all these values of $k$, we will have also covered all components of $Q_{m,M_{k_1}}^{-1}({\mathcal H}_{M_{k_1}})$ (except $\Hmt$) and it will then be easy to argue that we will be done.

So suppose for the base case that $K$ is a component of $Q_{m,M_{k_1}}^{-1}({\mathcal H}_{M_{k_1}})$ for which the joining time is $M_{k_0}$. From \emph{3.} of Claim \ref{JoiningTimes} we then know that $Q_{m, M_{k_0}-1}(K) \subset G_{M_{k_0}-1}^{k_0}$ while by \emph{2.} of Lemma \ref{HmtInvariance1} and Definitions \ref{GmkHmkDef}, \ref{HmtDef}, we know that $Q_{m, M_{k_0}-1}(\Hmt) = \tilde {\mathcal H}_{M_{k_0}-1} \subset H_{M_{k_0}-1}^{k_0}$. 

Using \emph{1.} of Lemma \ref{rksk} and Claim \ref{HmkConn}, the two connected sets $Q_{m, M_{k_0}-1}(K)$ and $ \tilde {\mathcal H}_{M_{k_0}-1}$ then lie in different complementary components of the (round) annulus $A_{k_0} = {\mathrm A}(\sqrt{-b_{k_0}}, r_{k_0}, 2\sqrt{-b_{k_0}} - r_{k_0})$ and, if we take a preimage under $Q_{m,M_{k_0}-1}$ of the corresponding disc ${\mathrm D}(\sqrt{-b_{k_0}}, 2\sqrt{-b_{k_0}}-r_{k_0})$, we have using Claim \ref{InverseBranches} and \emph{3.} of Claim \ref{JoiningTimes} that each of the components of $Q_{m,M_{k_1}}^{-1}(\tilde {\mathcal H}_{M_{k_1}})$ whose joining time is $M_{k_0}$ lies in the bounded complementary component of precisely one of the (pairwise disjoint) annuli $B_m^{k_0,j}$, $1 \le j \le 2^{M_{k_0} - m - 1}$. On the other hand, $ \tilde {\mathcal H}_{M_{k_0}-1}$ lies outside $A_{k_0}$ and (since the preimages of disjoint sets are disjoint), by \emph{3.} of Lemma \ref{HmtInvariance1}, 
$\Hmt = (Q_{m, M_{k_0}-1})^{-1}( \tilde {\mathcal H}_{M_{k_0}-1})$ lies outside the preimage under $Q_{m, M_{k_0}-1}$ of the disc ${\mathrm D}(\sqrt{-b_{k_0}}, 2\sqrt{-b_{k_0}}-r_{k_0})$ and thus in the unbounded complementary components of the annuli $B_m^{k_0,j}$, $1 \le j \le 2^{M_{k_0} - m - 1}$. Thus we have shown that $\Hmt$ and those components of 
$Q_{m,M_{k_1}}^{-1}(\tilde {\mathcal H}_{M_{k_1}})$ whose joining time is $M_{k_0}$ are separated by the collection of conformal annuli $B_m^{k_0,j}$, $1 \le j \le 2^{M_{k_0} - m - 1}$ in the sense that, given any two distinct components $K_1$, $K_2$ of $Q_{m,M_{k_1}}^{-1}(\tilde {\mathcal H}_{M_{k_1}})$ whose joining time is $M_{k_0}$ or $m$ (which is the only earlier joining time), these will be separated by one of the annuli $B_m^{k_0,j}$ above in the sense that one of these components will lie in the bounded and the other in the unbounded complementary component of this conformal annulus. 

For the induction hypothesis suppose the result is now true for some $k_0 \le k < k_1$, namely that the conformal annuli $B_m^{i,j}$, $k_0 \le i \le k$, $1 \le j \le 2^{M_i - m - 1}$ are pairwise disjoint and that this collection separates all the components of $Q_{m,M_{k_1}}^{-1}(\tilde {\mathcal H}_{M_{k_1}})$ whose joining time is at most $M_k$ (including $\Hmt$) in the sense that if $K_1$ and $K_2$ are any two distinct such components, then for some $B_m^{i,j}$, one will lie in the bounded and the other in the unbounded complementary component of this conformal annulus. 

We now turn to the induction step. By taking preimages of the disc ${\mathrm D}(\sqrt{-b_{k+1}}, 2\sqrt{-b_{k+1}}-r_{k+1})$ at time $M_{k+1}-1$ under $Q_{M_k-1, M_{k+1}-1}$ and applying the backward invariant condition \eqref{DisjointAnnuli1} and Definition \ref{AkDef}, we have that 

\begin{eqnarray*}
(Q_{M_k-1, M_{k+1}-1})^{-1}(A_{k+1}) &=& \\
(Q_{M_k-1, M_{k+1}-1})^{-1}(\!\!\!\!\!&{\mathrm A}&\!\!\!\!\!(\sqrt{-b_{k+1}}, r_{k+1}, 2\sqrt{-b_{k+1}} - r_{k+1}))\\ \subset& &\hspace{-.8cm}(G_{M_k-1}^k \cup H_{M_k -1}^k).
\end{eqnarray*}

 On taking a further preimage under $Q_{m,M_k-1}$ (in stages $Q_{M_{i-1}-1,M_i-1}$ starting with $Q_{M_{k-1}-1,M_k-1}$ and lastly $Q_{m,M_{k_0}-1}$) and applying \eqref{DisjointAnnuli1} repeatedly, we then see that the conformal annuli $B_m^{k+1,j}$, $1 \le j \le 2^{M_{k+1} - m - 1}$ are disjoint from our earlier collection $B_m^{i,j}$, $k_0 \le i \le k$, $1 \le j \le 2^{M_i - m - 1}$. In addition, using Claim \ref{InverseBranches} these annuli are also disjoint from each other so that the whole collection $B_m^{i,j}$, $k_0 \le i \le k+1$, $1 \le j \le 2^{M_i - m - 1}$ is now pairwise disjoint.

\afterpage{\clearpage}

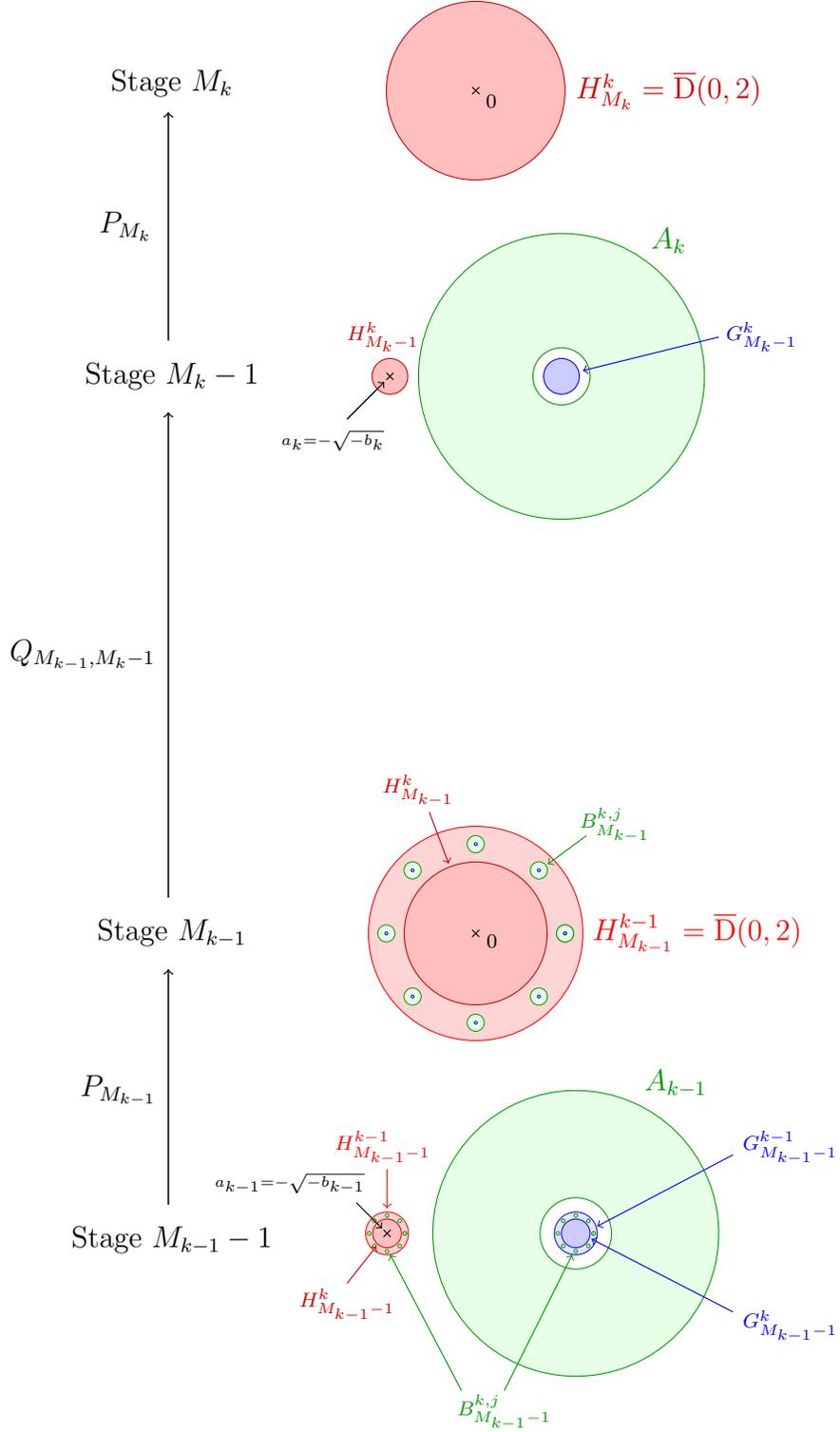
\begin{figure}
\vspace{.5cm}
\begin{tikzpicture}

\node at (-3.45,11.1) {Stage $M_k$};

\node at (-4.1, 9.1) {$P_{M_k}$};

{\color{darkred}
\draw[fill=lesspalepink] (.8,11) circle (1.25);}

 {\color{black}
  \draw[line width=0.15mm] (.75,10.95) -- (.85,11.05);
  \draw[line width=0.15mm] (.75,11.05) -- (.855,10.95);
  
  \node at (1.02,10.85) {$\scriptstyle 0$};
  
  \draw[line width=0.15mm, ->] (-1.0,6.4) -- (-.48,6.92);
  }

\node at (3.5, 11) {\color{darkred}$H^k_{M_k} = \overline {\mathrm{D}}(0,2)$};

\node at (-3.45,7.0) {Stage $M_k-1$};

\draw[line width=0.2mm, ->] (-3.5,7.5) -- (-3.5,10.7);

{\color{darkgreen}
\draw[fill=palegreen] (2,7) circle (2);
\draw[fill=white] (2,7) circle (.4);}

\node at (3.5, 8.9) {\color{darkgreen}$A_k$};

  \node at (4.8, 7.55) {\color{darkblue}$\scriptstyle  G^k_{M_k-1}$};
  
{\color{darkblue}
\draw[line width=0.15mm, ->] (4.2,7.55) -- (2.3,7.08);
}

  {\color{darkblue}
  \draw[fill=lesspaleblue] (2,7) circle (.25);}
  
  \node at (-.5, 7.55) {\color{darkred}$\scriptstyle  H^k_{M_k-1}$};

   {\color{darkred}
  \draw[fill=lesspalepink] (-.4,7) circle (.25);}
  
  {\color{black}
  \draw[line width=0.15mm] (-.45,6.95) -- (-.35,7.05);
  \draw[line width=0.15mm] (-.45,7.05) -- (-.35,6.95);
  
  \node at (-1.2,6.1) {$\scriptscriptstyle a_k = - \sqrt{-b_k}$};
  
  \draw[line width=0.15mm, ->] (-1.0,6.4) -- (-.48,6.92);
  }
  
  {\color{red}
\draw[fill=palepink] (.8,-.8) circle (1.5);}

 \node at (3.9, -.8) {\color{red}$H^{k-1}_{M_{k-1}} = \overline {\mathrm{D}}(0,2)$};

{\color{darkred}
\draw[fill=lesspalepink] (.8,-.8) circle (1.00);}

  \node at (0, 1.2) {\color{darkred}$\scriptstyle  H^k_{M_{k-1}}$};
 
  {\color{black}
  \draw[line width=0.15mm] (.75,-.75) -- (.85,-.85);
  \draw[line width=0.15mm] (.75,-.85) -- (.85,-.75);
  
  \node at (1.03,-.902) {$\scriptstyle 0$};
 
  }   
  
{\color{darkred}
\draw[line width=0.15mm, ->] (.15,0.9) -- (.42,.18);
}

\node at (2.75, 0.72) {\color{darkgreen}$\scriptstyle  B^{k,j}_{M_{k-1}}$};

{\color{darkgreen}
\draw[line width=0.15mm, ->] (2.35,0.57) -- (1.81,0.17);
}

\node at (-3.45,-.8) {Stage $M_{k-1}$};

\node at (-4.7, 3.1) {$Q_{M_{k-1},M_k -1}$};

\draw[line width=0.2mm, ->] (-3.5,-.3) -- (-3.5,6.5);

{\color{darkgreen}
  \foreach \phi in {0,45,...,360}{
    \draw[line width=0.1mm,fill=palegreen] ({.8+1.25*cos(\phi)},{-.8+1.25*sin(\phi)}) circle (.12);
  }}

  {\color{blue}
  \foreach \phi in {0,45,...,360}{
    \draw[line width=0.1mm, fill=paleblue] ({.8+1.25*cos(\phi)},{-.8+1.25*sin(\phi)}) circle (.02);
  }}
  
 \node at (-4.2, -3) {$P_{M_{k-1}}$};
  
  \draw[line width=0.2mm, ->] (-3.5,-4.6) -- (-3.5,-1.3);
  
  \node at (-3.45,-5.1) {Stage $M_{k-1}-1$};
  
 {\color{darkgreen} 
 \draw[fill=palegreen] (2.2,-5) circle (2);
\draw[fill=white] (2.2,-5) circle (.5); }

\node at (3.6, -2.9) {\color{darkgreen}$A_{k-1}$};

{\color{blue}
  \draw[line width=0.05mm, fill=paleblue] (2.2,-5) circle (.3);}

 {\color{darkblue}
  \draw[line width=0.05mm, fill=lesspaleblue] (2.2,-5) circle (.2);}
  
  \node at (5.2, -3.8) {\color{blue}$\scriptstyle G^{k-1}_{M_{k-1}-1}$};
  
  {\color{blue}
\draw[line width=0.15mm, ->] (4.4,-3.9) -- (2.5,-4.9);
}

  \node at (5.2, -6.3) {\color{darkblue}$\scriptstyle G^k_{M_{k-1}-1}$};
  
    {\color{darkblue}
\draw[line width=0.15mm, ->] (4.4,-6.1) -- (2.42,-5.08);
}
  
  {\color{darkgreen}
  \foreach \phi in {0,45,...,360}{
    \draw[line width=0.05mm, fill=palegreen] ({2.2+.25*cos(\phi)},{-5+.25*sin(\phi)}) circle (.025);
  }}
  
  \node at (1.2, -7.5) {\color{darkgreen}$\scriptstyle B^{k,j}_{M_{k-1}-1}$};
  
{\color{darkgreen}
\draw[line width=0.15mm, ->] (1.2,-7.2) -- (2.15, -5.3);
}
  
       {\color{red}
  \draw[line width=0.05mm, fill=palepink] (-.44,-5) circle (.3);}
  
     {\color{darkred}
  \draw[line width=0.05mm, fill=lesspalepink] (-.44,-5) circle (.2);}
  
   {\color{darkgreen}
  \foreach \phi in {0,45,...,360}{
    \draw[line width=0.05mm, fill=palegreen] ({-.44+.25*cos(\phi)},{-5+.25*sin(\phi)}) circle (.025);
  }}
  
\node at (-.5, -3.8) {\color{red}$\scriptstyle H^{k-1}_{M_{k-1}-1}$};

{\color{red}
\draw[line width=0.15mm, ->] (-.44,-4.1) -- (-0.44, -4.65);
}
  
 \node at (-1.0, -6.0) {\color{darkred}$\scriptstyle H^k_{M_{k-1}-1}$};
 
{\color{darkred}
\draw[line width=0.15mm, ->] (-1.0,-5.7) -- (-0.59, -5.18);
}
 
 {\color{darkgreen}
\draw[line width=0.15mm, ->] (0.6,-7.2) -- (-0.395, -5.3);
}

 {\color{black}
  \draw[line width=0.15mm] (-.49,-4.95) -- (-.39,-5.05);
  \draw[line width=0.15mm] (-.49,-5.05) -- (-.39,-4.95);
  
  \node at (-1.8,-4.3) {$\scriptscriptstyle a_{k-1} = - \sqrt{-b_{k-1}}$};
  
  \draw[line width=0.15mm, ->] (-.87,-4.57) -- (-.51,-4.93);
  }
 
\end{tikzpicture}   
 \caption{The setup for Claim \ref{SeparatingAnnuliConfiguration}}  \label{SeparatingAnnuliConfigurationPicture}
\end{figure}

\newpage

For any component $K_1$ of $Q_{m,M_{k_1}}^{-1}(\tilde {\mathcal H}_{M_{k_1}})$ whose joining time is $M_{k+1}$, we must have by \emph{3.} of Claim \ref{JoiningTimes} that $Q_{m, M_{k+1}-1}(K_1) \subset G_{M_{k+1}-1}^{k+1}$. On the other hand, for any component $K_2$ of this set whose joining time is strictly earlier than $M_{k+1}$, by \emph{2.} of Claim \ref{JoiningTimes}, \emph{2.} of Lemma \ref{HmtInvariance1}, and Definitions \ref{GmkHmkDef}, \ref{HmtDef}, we must have $Q_{m, M_{k+1}-1}(K_2) \subset H_{M_{k+1}-1}^{k+1}$. 
On again taking the preimage of the disc ${\mathrm D}(\sqrt{-b_{k+1}}, 2\sqrt{-b_{k+1}}-r_{k+1})$ at time $M_{k+1}-1$ under $Q_{m, M_{k+1}-1}$, we see using Claim \ref{InverseBranches} that each of those components of $Q_{m,M_{k_1}}^{-1}(\tilde {\mathcal H}_{M_{k_1}})$ whose joining time is $M_{k+1}$ will lie in the bounded complementary component of precisely one of the pairwise disjoint conformal annuli $B_m^{k+1,j}$, $1 \le j \le 2^{M_{k+1} - m - 1}$ while those whose joining time is strictly earlier than $M_{k+1}$ will lie in the unbounded complementary components of these annuli. Combining this with the induction hypothesis allows us to deduce easily that the separating condition for $k+1$ follows, namely that the collection of conformal annuli $B_m^{i,j}$, $k_0 \le i \le k+1$, $1 \le j \le 2^{M_i - m - 1}$ separates all the components of $Q_{m,M_{k_1}}^{-1}(\tilde {\mathcal H}_{M_{k_1}})$ whose joining time is at most $M_{k+1}$ in the same sense as given above. With this, the induction step is complete. 

Continuing in this way until $k = k_1$ and remembering from Claim \ref{JoiningTimes} that this is the latest possible joining time so that all components of $Q_{m,M_{k_1}}^{-1}(\tilde {\mathcal H}_{M_{k_1}})$ are accounted for, we have thus proved Claim \ref{SeparatingAnnuliConfiguration}.  \end{claimproof}

We showed above shortly after introducing the conformal annuli $B^{k,j}_m$ above in Definition \ref{BkjmDef} that these annuli all lie in $\Am$ and it follows from Claim \ref{SeparatingAnnuliConfiguration} that these annuli are all disjoint, so that by the claim above, all components of $Q_{m,M_{k_1}}^{-1}(\tilde {\mathcal H}_{M_{k_1}})$ (including $\Hmt$) are separated from each other by conformal annuli which lie in the basin of infinity (and since these annuli lie in $\Am$, it is impossible for any points of $\Jm$ to `intrude' into them as $n$ tends to infinity). 

Since by part \emph{4.} of Claim \ref{Hmt} $\partial \Hnt$ is connected, if $k_0$ is as small as possible so that $m \le M_{k_0}$, on letting $k$ go to infinity, using \eqref{HmDef1} we then obtain that $\Hm  = \cup_{k=k_0}^\infty{Q_{m,M_k}^{-1}(\partial \tilde{\mathcal H}_{M_k})}$ does indeed consist of countably infinitely many connected components which from above lie in different components of $\Jm$, which finishes the proof of \emph{5.} of the statement of Theorem \ref{MainTh1}.

We next prove \emph{1.} in the statement of Theorem \ref{MainTh1}. As stated above in Claim \ref{Hm1}, $\Hm$ is precisely the set of points $z \in \Jm$ for which $Q_{m, M_k-1}(z) \in H^k_{M_k -1}$ for all but finitely many $k$. Theorem \ref{ThmCompInvar} then tells us that the sets $\Hm$ give us a completely invariant collection. Next, for each $m \ge 0$ we define

\begin{equation}
\label{GmDef1}
\Gm =  \Jm \setminus \Hm. \vspace{.2cm}
\end{equation}

It then follows from Theorem \ref{ThmCompInvar} and the complete invariance of the sets $\Hm$ that the sets $\Gm$ also give us a completely invariant collection. 

In addition, by \emph{1.} of Lemma \ref{GmkHmkDecr} the sets $G^k_m$ are nested, compact, and non-empty so that, if $k_0$ is as small as possible so that $M_{k_0}  - 1 \ge m$, then $\cap_{k \ge k_0} G^k_m \neq \emptyset$. However, by \emph{2.} of this same lemma, this is contained in the set of points $z$ for which $Q_{m,M_k-1}(z) \in G^k_{M_k -1}$ for all $k \ge k_0$. 

By Claim \ref{Hm1} $\Gm$ is precisely the set of points $z \in \Jm$ for which $Q_{m, M_k-1}(z) \notin H^k_{M_k -1}$ for infinitely many $k$. However, by \eqref{Invariance1}, \eqref{Invariance2} and Lemma \ref{JSubseq}, it follows that the iterates $Q_{M_k-1,n}$ escape uniformly to infinity on $\chat \setminus (G^k_{M_k -1} \cup H^k_{M_k -1})$ so that ${\mathcal J}_{M_k -1} \subset G^k_{M_k -1} \cup H^k_{M_k -1}$. From this, we can then see that $\Gm$ is also precisely the set of points $z \in \Jm$ for which $Q_{m, M_k-1}(z) \in G^k_{M_k -1}$ for infinitely many $k$. It follows that $\cap_{k \ge k_0} G^k_m \subset \Gm$ whence the set $\Gm$ is non-empty.

Recalling that we already observed that $\Hm$ is non-empty (shortly after defining this set in \eqref{HmDef1}), 
\emph{1.} in the statement of Theorem \ref{MainTh1} then follows. Also, since $\Hm$ is $\Fsigma$, $\Gm$ is then automatically a (relatively) $\Gdelta$ subset of $\Jm$. Since $\Gm$ is then precisely the set of points $z \in \Jm$ for which $Q_{m, M_k-1}(z) \in G^k_{M_k -1}$ for infinitely many $k$, it follows that, if $z \in \Gm$, then so is any point in $Q_{m,n}^{-1}(Q_{m,n}(z))$ for any fixed $n \ge m$.  Applying \emph{1.} of B\"uger's result (Theorem \ref{SelfSimilarity}) combined with Lemma \ref{JSubseq}, as we did above for the sets $\Hm$, it then follows $\Gm$ is then dense in $\Jm$. 


It now remains to establish \emph{2.} and \emph{4.} of the statement of Theorem \ref{MainTh1}. Recall from Claim \ref{InverseBranches} above that, if $0 \le m \le M_k -1$, 
each of the polynomials $Q_{m, M_k -1}$ has a full set of $2^{M_k - m - 1}$ inverse branches defined on the disc ${\mathrm D}(\sqrt{-b_k}, 2\sqrt{-b_k})$. In addition, if $M_{k-1} \le m \le M_k -1$, then these preimages differ by rotations of $\tfrac{2\pi}{2^{M_k - m - 1}}$. Recall also the (round) annuli $A_k={\mathrm A}(\sqrt{-b_k}, r_k, 2\sqrt{-b_k} - r_k) = {\mathrm A}(-a_k, r_k, 2\sqrt{-b_k} - r_k)$ at time $M_k -1$ as defined in Definition \ref{AkDef} shortly before \eqref{InverseImageofDisc}. In view of \emph{1.} of Lemma \ref{rksk}, $A_k$ then contains $G^k_{M_k-1}$ and $H^k_{M_k-1}$ in its bounded and unbounded complementary components respectively, so that it separates these sets. From above, for each $0 \le m \le M_k -1$, the inverse image of this annulus under $Q_{m, M_k -1}$ consists of $2^{M_k - m - 1}$ conformal annuli of the same modulus as $A_k$
which we denoted by $B^{k,j}_m$, $1 \le j \le 2^{M_k - m - 1}$ in Definition \ref{BkjmDef}. Again, as noted earlier, by Lemma \ref{rkskbounds} the modulus of $A_k$ is $\log \tfrac{ 2\sqrt{-b_k} - r_k}{r_k}  > \log (4\sqrt{-b_k} - 1)$ which then tends to infinity as $k \to \infty$ as $b_k \to - \infty$ as $k \to \infty$.

To use these annuli $B^{k,j}_m$ for pointwise thinness, we need to show that their Euclidean diameters tend to $0$ as $k$ tends to infinity. More precisely we have: 

\begin{claim}
\label{ShrinkingSeparatingAnnuli1}
For each $m \ge 0$, if $k_0$ is the smallest integer such that $m \le M_{k_0}-1$ then, for each $k \ge k_0$ and each $1 \le j \le 2^{M_k - m - 1}$, 
$$ \diam B^{k, j}_m \le \frac{(4\pi + \tfrac{3}{2})\,6\sqrt{-b_{k_0}}}{2^{k - k_0}}.$$
\end{claim}

\begin{claimproof}
By Definition \ref{AkDef} and \eqref{mkbound1} of Claim \ref{InvariantAnnuliDiscs} (or the stronger \eqref{mkbound2} of Claim \ref{Contracting1}), we have $Q_{M_{k-1}, M_k -1}(\overline {\mathrm D}(0,2)) \supset \overline {\mathrm D}(0,8\sqrt{-b_k}) \supset A_k$ and, since $Q_{M_{k-1}, M_k -1} = z^{2^{m_k}}+a_k$ maps $ {\mathrm D}(0,2)$ to its image with the full degree $2^{m_k}$ of this polynomial, it follows that $Q_{M_{k-1}, M_k -1}^{-1}(A_k) \subset \overline {\mathrm D}(0,2)$. On the other hand, again by Definition \ref{AkDef}, the distance of $A_k$ from the critical value $a_k$ of $Q_{M_{k-1}, M_k -1}$ is $r_k$ and by Lemma \ref{rkskbounds} $r_k \ge \tfrac{1}{\sqrt{-b_k}}$. Lastly we then have by \eqref{InnerRadiusBound1} in the proof of Claim \ref{InvariantAnnuliDiscs} that 

$$\frac{1}{2} \le \left ( \frac{1}{8\sqrt{-b_k}}\right)^{1/2^{m_k}} < \: \left ( \frac{1}{\sqrt{-b_k}}\right)^{1/2^{m_k}} < r_k^{1/2^{m_k}}$$

whence we obtain $Q_{M_{k-1}, M_k -1}^{-1}(A_k) \subset \overline {\mathrm A}(0,\tfrac{1}{2},2)$. On the other hand, since $A_k$ lies entirely to the right of $a_k = -\sqrt{-b_k}$ which is the critical value of $Q_{M_{k-1}, M_k -1}$, it follows that each component of the preimage lies in a sector ${\mathrm S}(0, \theta_1, \theta_2)$ at $0$ of aperture $\tfrac{\pi}{2^{m_k}} \le \pi$. 

The fact that $Q_{M_{k-1}, M_k -1}^{-1}(A_k) \subset \overline {\mathrm A}(0,\tfrac{1}{2},2)$ from above allows us to estimate the size of the annuli $B^{k,j}_m$, $1 \le j \le 2^{M_k - m - 1}$ by estimating the derivatives of the inverse branches of $Q_{m, M_{k-1}}$ from the times $M_{k-1}$ on the annuli $B^{k,j}_{M_{k-1}}$, $1 \le j \le 2^{m_k}$ which are then subsets of $\overline {\mathrm A}(0,\tfrac{1}{2},2) \subset {\mathrm A}(0,\tfrac{1}{3},3)$, instead of estimating the derivatives of the inverse branches of $Q_{m, M_k -1}$ from the times $M_k -1$  on the annuli $A_k$ (we direct the reader to carefully note the difference between the times $M_{k-1}$ and $M_k -1$). Using  \emph{1.} of Claim \ref{InvariantAnnuliDiscs} and both parts of Claim \ref{Contracting1}, we have that the diameter of each of the conformal annuli $B^{k,j}_m$, $1 \le j \le 2^{M_k - m - 1}$ above is then at most $\tfrac{\left (4\pi + \tfrac{3}{2} \right )6\sqrt{-b_{k_0}}}{2^{k - k_0}}$, where the quantity $4\pi + \tfrac{3}{2}$ is a crude estimate for the distance between points inside ${\mathrm A}(0,\tfrac{1}{2},2)$ where we connect them with a path in this annulus which lies inside a sector based at $0$ as above of aperture at most $\pi$ and which consists of arcs and radial line segments. The desired estimate then follows and we have thus proved Claim \ref{ShrinkingSeparatingAnnuli1}.
\end{claimproof}

Using Lemma \ref{rkskbounds} and Definition \ref{AkDef}, the set $G_{M_k-1}^k$ at time $M_k-1$ lies in the bounded complementary component of the annulus $A_k = {\mathrm A}(\sqrt{-b_k}, r_k, 2\sqrt{-b_k} - r_k)$. Recalling from Claim \ref{InverseBranches} above that  $Q_{M_k -1}$ has a complete set of univalent inverse branches defined on the disc ${\mathrm D}(\sqrt{-b_k}, 2\sqrt{-b_k})$, using Definition \ref{BkjmDef}
each of the conformal annuli $B^{k,j}_m$ contains precisely one of the components of $Q_{m,M_k-1}^{-1}(G_{M_k-1}^k)$ in its bounded complementary component
and thus separates this component from the others (as well as from the connected set $H^k_m$).  

Now let $m \ge 0$ and let $z \in \Gm$ be arbitrary. Recall that we observed earlier (two paragraphs after the definition of the sets $\Gm$ in \eqref{GmDef1}) that $\Gm$ is precisely the set of points for which $Q_{m, M_k-1}(z) \in G^k_{M_k -1}$ for infinitely many $k$ whence from the above $z$ lies in the bounded complementary component of one of the annuli in the collection $B^{k,j}_m$, $1 \le j \le 2^{M_k - m - 1}$ for infinitely many $k$. Since the Euclidean diameters of these conformal annuli shrink to $0$ as $k \to \infty$ by Claim \ref{ShrinkingSeparatingAnnuli1}, while their moduli tend to infinity (as remarked shortly after the statement of Claim \ref{InverseBranches}), it then follows that these annuli will separate $\Gm$ for $k$ sufficiently large and that $\Gm$ is thus pointwise thin and, again from our earlier observations (in the paragraphs following \eqref{GmDef1}), a dense $\Gdelta$ subset of $\Jm$. We remark that this is where we need $b_k \to -\infty$ as $k \to \infty$ since otherwise if we merely had an unbounded sequence, the values of $k$ above where $z$ lies in the bounded complementary component of one of the annuli in the collection $B^{k,j}_m$ might correspond to a bounded subsequence of $\{b_k\}_{k=1}^\infty$ which would not allow us to conclude that $\Gm$ is pointwise thin at $z$. Finally, since by \eqref{GmDef1} above, $\Jm \setminus \Gm = \Hm$ which we already know by \emph{5.} is not pointwise thin at any of its points, it follows that $\J_{m, \mbox{\scriptsize thin}} = \Gm$, which finishes the proof of \emph{2.}

Finally, since we showed in proving \emph{5.} (just after the proof of Claim \ref{Hm1}) that $\Hm$ is dense in $\Jm$, we have $\overline{\mathcal H}_{m}= \Jm$.
However, $\Jm$ is not uniformly perfect since we showed in proving \emph{1.} above  that ${\mathcal J}_{m, \mbox{\scriptsize thin}}$  is a non-empty subset of $\Jm$. Thus $\overline{\mathcal H}_{m}$ is not uniformly perfect, which completes the proof of Theorem \ref{MainTh1}. \hfill $\Box$


\end{document}